\documentclass[a4paper,10pt]{article}

\usepackage{geometry}
\usepackage{amsmath}
\usepackage{amssymb}
\usepackage{amsthm}
\usepackage[latin1]{inputenc}
\usepackage{eurosym}
\usepackage[dvips]{graphics}
\usepackage{graphicx}
\usepackage{epsfig}

\usepackage{hyperref}

\usepackage{ifthen}

\newcommand{\Rr}{{\mathbb{R}}}

\newcommand{\Nn}{{\mathbb{N}}}

\newcommand{\rar}{\rightarrow}
\newcommand{\bi}{{\bold{i}}}

\newcommand{\bk}{{\bold{k}}}
\newcommand{\bn}{{\bold{n}}}

\newcommand{\Sss}{{\mathcal{S}}}
\newcommand{\Sdn}{{\mathcal{S}^d_N}}

\newcommand{\cqd}{\hfill $\blacksquare$}

\newtheorem{theorem}{Theorem}

\newtheorem{lemma}{Lemma}
\newtheorem{proposition}{Proposition}
\newtheorem{definition}{Definition}
\newtheorem{assumption}{Assumption}

 \newcommand{\re}{\mathbb{R}}

\def\cqd {\,  \begin{footnotesize}$\square$\end{footnotesize}}

\begin{document}

\title{Continuous time finite state mean field games }
\author{Diogo
  A. Gomes
\footnote{Departamento de Matem\'atica and CAMGSD, IST, Lisboa, Portugal. e-mail: dgomes@math.ist.utl.pt}
, Joana Mohr \footnote{Instituto de Matem\'atica, UFRGS, 91509-900 Porto Alegre, Brasil. e-mail: rafars@mat.ufrgs.br}
and Rafael Rig\~ao Souza \footnote{Instituto de Matem\'atica, UFRGS, 91509-900 Porto Alegre, Brasil. e-mail: joana.mohr@ufrgs.br}
}

\date{\today} 

\maketitle

\begin{abstract}

In this paper we consider  symmetric games where a large number of players can be in any one of d states. We derive a limiting mean field model and characterize its main properties.
This mean field limit is a system of coupled ordinary differential equations with initial-terminal data.
For this mean field problem we prove a trend to equilibrium theorem, that is convergence, in an appropriate limit, to stationary
solutions.
Then we study the $N+1$-player problem, which the mean field model attempts to approximate.
Our main result is the convergence as $N\to \infty$ of the mean field model and an estimate of the rate of convergence.
We end the paper with some further examples for potential mean field games.
\end{abstract}

\thanks{D. Gomes was partially supported by CAMGSD-LARSys through FCT, POCTI-FEDER, and by grants PTDC/MAT/114397/2009,
UTAustin/MAT/0057/2008, and UTA-CMU/MAT/0007/2009, and by the bilateral agreement Brazil-Portugal (CAPES-FCT) 248/09}

\thanks{R.R.S was partially supported by the bilateral agreement Brazil-Portugal (CAPES-FCT) 248/09}

\thanks{J.M was partially supported by the bilateral agreement Brazil-Portugal (CAPES-FCT) 248/09.}

\section{Introduction}

Mean field games is a recent area of research started by Peter Caines and his co-workers \cite{Caines1},
\cite{Caines2}, and independently  by Pierre Louis Lions and Jean Michel Lasry
\cite{ll1, ll2, ll3, ll4} which attempts to understand the limiting behavior of systems involving very large
numbers of rational agents which play dynamic games under partial information and symmetry assumptions.
Inspired by ideas in statistical physics, in this
class of models the individual player's contributions are encoded into a mean field
that contains all relevant statistical properties about the ensemble.

The literature on mean field games and its applications
is growing fast. For recent surveys see \cite{llg2} or \cite{pc}, and reference therein. Mean field games
arise in the study of growth theory in economics \cite{llg1, ML}, production of exaustible resources \cite{llg2}, or environmental policy \cite{lst}, for instance, and it is likely that in
the future they will play an important role in economics and population models. There
is also a growing interest in numerical methods for
these problems \cite{lst}, \cite{DY}. A related concept,
called oblivious equilibrium, corresponds to the case where players are assumed to make decisions based only
on its own state and knowledge of the long-run average industry state and stationary equilibrium models were introduced and studied in
detail in, respectively,
\cite{weintraub1} and \cite{weintraub2}.
Mean field models correspond to the limit of $N$ player games
under symmetry assumptions. The Markov perfect equilibrium notion for these games
has been studied (mostly in discrete time or stationary setting) in \cite{mpe1, mpe3, mpe4, mpe12}, and references therein. In
\cite{mpe5, mpe6} symmetric Markov perfect equilibrium are also considered, and in the last paper the case with an infinite number of players
is studied. In \cite{mpe11} the passage from discrete time to continuous time is considered for $N$ players in a war of attrition problem.
The techniques in the present paper are, however, substantially different from the above references.

In this paper we begin by presenting a mean field model for a continuous time  dynamic game between a large number of rational agents, which we call players.
These
are allowed to switch between a finite number of states, looking forward to optimize certain functionals, which depend on the statistical distribution of
the other players. We discuss the concept of Nash equilibrium, which allow us to derive a system of ordinary differential equations for
the distribution of the players, as well their value function.
 After, we consider the $N+1$-player game, which corresponds to the previous problem before taking the mean field limit. In the $N+1$-player game each player knows only the state but not the identity of the remaining players.
  We are particularly interested in understanding
the limit of the $N+1$ player game as the number of players increases to infinity.

In discrete time, finite number of state mean field models were studied in \cite{GMS}.
In his PhD thesis, \cite{GueantT}, O. Gu\'eant
considered a problem with two states, modeling the labor market.
In this work he considered a continuum of individuals and a  labor market consisting of 2
sectors. Each individual has to decide on which sector
he or she is going to work.
This model consists in a coupled
systems of ordinary differential equations of the type that will be derived in \S\ref{mfmsec}.
More recently in \cite{Gueant1}, and \cite{Gueant2} several discrete state problems have been also studied in detail,
namely its connection with systems of conservation laws.
Further
models with discrete state space were also considered in \cite{tembine9}, \cite{tembine2}. In these models,
 each individual in a large population
interacts with randomly selected players. This
 interaction determines the instantaneous payoff for all involved players. In particular these authors establish several very interesting limit
results. We should note, however that these last works do not study mean field games in the sense of this paper, namely they
lack the forward-backward structure of the equilibrium as in the works by Caines, Lions-Lasry, among others.

We start in \S\ref{mfmsec} by describing the mean field game. We derive a mean field model for the optimal switching policy of a reference player
given the fraction $\theta(t)\in [0,1]^d$ of players in each of $d$ states. Then we introduce the concept
of Nash equilibrium.
This equilibrium turns out to be determined by
a coupled system of ordinary differential equations, where one equation governs the evolution of
$\theta$, and is subjected to initial conditions, whereas the other equation models
the evolution of a value function and has terminal data. We call this problem the initial-terminal
value problem. These models are similar to the ones in \cite{Gueant1}, \cite{Gueant2}.
Initial terminal value problems are in fact a general feature in many mean field game
problems, see for instance \cite{ll1, ll2, ll3}, though not very common in ODE problems. In fact, existence and uniqueness of solutions
is not immediate from the general ODE theory but, adapting the methods of Lions and Lasry we were successful
in establishing both.
We also study a class of contractive mean field games for which a-priori bounds on suitable norms can be established. In particular under this condition one can prove existence of stationary solutions.
 The main result of this section is a trend to equilibrium theorem, in the spirit of
the results in \cite{GMS}. The proof relies on a reverse Gronwall inequality (i.e. when an integral of a function is controlled by the function at the endpoints).

In \S\ref{nplayer} we consider the Nash equilibrium problem before taking a mean field limit, i.e. with a finite number, $N+1$, of players. 
As before we suppose that all players are identical and so the game is symmetric with respect to permutation of the players. We adopt the point of view of a reference
player, which could be chosen as any one of the players. 
We assume that this player (as any other) has access to the same information, namely, his/her own state at time $t$, given by $\bi_t\in \{1,2,3,...,d\}$,
and the number
$\bn_t\in \Nn^d$ of remaining players that are in the other states.
The objective of the reference player is to minimize, by controlling the process $\bi_t$, and given the
process $\bn_t$, the expected value of the integral of a running cost function added to a terminal cost.
We assume that both $\bi_t$ and $\bn_t$ are controlled non-time homogeneous coupled Markov chains. 
More precisely, we suppose that $N$ of
the players have a fixed Markov switching strategy $\beta$, known by the reference player, which then chooses a
switching strategy $\alpha(\beta)$. This is a well know
Markov decision problem. The Nash equilibrium corresponds to $\alpha(\beta)=\beta$, which can be characterized by a
system of ordinary
differential equations.
In this setting the equilibrium is characterized by a system of ordinary differential equations with a terminal condition. This system, as explained in \S\ref{lastsection}, can be seen as a discretized version of a partial differential equation (introduced by Lions in his course in College de France and further studied by Gueant \cite{Gueant1, Gueant2})
for the value function that can be derived for the mean field model, as an alternative to the initial-terminal value problem formulation.
In addition to this characterization we prove various bounds, uniformly on $N$, which then allow
to address the passage to the limit problem, in \S\ref{convsec}.

In \S\ref{convsec} we prove the main result of the paper, Theorem \ref{teoconv},
which is
the convergence as the number of players $N\to \infty$ in $L^2$ of the $N+1$-player model to the mean field model of \S\ref{mfmsec}. In a different setting, convergence to MFG model was established by \cite{KLY} using very interesting techniques from non-linear Markov chains. We should note that the techniques in that paper do not apply to the problem we consider were, as our problem has a different structure. Our convergence result, gives, for small $T$, a rate of convergence of the order $\frac 1 {\sqrt{N}}$. For the proof we not need monotonicity assumptions. In particular this implies
uniqueness of solution to the mean field problem for small time. Our proof uses a double Gronwall-type inequality where part of the integrand can be estimated forward in time, whereas other part can only be estimated backwards in time.

In \S\ref{lastsection} we end this paper with an important class of examples, namely potential mean field games.
These have been studied in detail by Pierre Louis Lions (College de France course) and also in \cite{Gueant1, Gueant2}.
For these mean field games several connections with Hamiltonian and Lagrangian dynamics can be derived which have interesting applications to planning problems. We also discuss a variational formulation in analogy to the results in \cite{GMorgado} and \cite{GPSM}, as well
as some connections with partial differential equations, numerical methods and Hamilton-Jacobi equations.

\section{A mean field model}
\label{mfmsec}

In this section we derive a mean field model which, as we will show later, corresponds to the limit as the number of players
tends to infinity
of symmetric dynamic games with a finite number of players.

We consider a continuous time dynamic game where a large number of players can be in any of $d$ states. The players can switch from state
to state and their decisions depend on certain
optimality criteria which we will describe in the following. We suppose that all players are identical and so the game is symmetric with
respect to permutation of the players. Players only know its own position and the fraction of players in each of the $d$ states.
Each player can control the transition rate from one state to another and incurs in both a running cost
and a terminal cost which depends on  its own state, on the state of the other players
(through its distribution among states and not on individual player's states)
as well as on the controls the player chooses.

We will fix one of the players which will be called the reference player.
Because the game is symmetric, the identity of this
player is not important, and all other players have access to similar information. 
We further assume the mean field hypothesis, that is, since the number of players
is very large, the only information available to the reference player is the
distribution of players given by a probability vector $\theta\in \Sss^d$, where $\mathcal{S}^d$ is the probability
simplex
$$\left\{
  \begin{array}{l}
    \theta^1 + ... + \theta^d = 1\,,\\
    \theta^i\geq 0 \quad \forall i,  \; 1\leq i \leq d\,.
  \end{array}
\right.$$
Under the mean field hypothesis, the evolution
of the vector $\theta$ can be approximated by an ordinary differential equation
as discussed in \S \ref{ctmpke}.





\subsection{Continuous time Markov process and the Kolmogorov equation}
\label{ctmpke}

We suppose that the players distribution among states is given by a probability vector $\theta(t) \in \mathcal{S}^d$.
Let $\beta(t) \in \re^{d \times d}$ represent a transition rate matrix depending on the time $t$, where $\beta_{ij}(t) \geq 0$ if $i \neq j$, and
 $\beta_{ii}\equiv-\sum_{j\neq i} \beta_{ij}\,.$
We assume that
 the players switch from state to state according to a continuous time (inhomogeneous) Markov process with transition rate matrix $\beta$, which for now we suppose it is known. In the mean field limit,
 the fraction of players in each state $\theta$ satisfies the Kolmogorov equation
\begin{eqnarray}
  \frac{d\theta^{i}}{dt}  &=& \sum_j \theta^j \beta_{ji}\,. \label{kolmogorov}
 \end{eqnarray}
The previous equation is complemented
by an initial condition $\theta(0)=\theta_0 \in \mathcal{S}^d$ from which the evolution of the distribution of players  $\theta^{\beta}:[0,T]\rar \mathcal{S}^d$
is completely determined.
For convenience,  controls are also identified with a vector $\beta(i)\in \re^d$ with the convention that $\beta_j(i)=\beta_{ij}$,
where $\beta_j(i)$ denotes the
$j-$th coordinate  of  $\beta(i)$\,.


\subsection{Running and terminal costs}

We fix now a reference player and consider the optimization problem according to his/her point of view.
We assume that the state of this player is driven by a continuous time discrete state optimal control
problem in which he/she controls the switching rates from state to state.
These switching rates are chosen in order to minimize a certain cost which is the sum of
a running cost and a terminal cost.  The running cost depends on
the player's state, the switching rate, and
the fraction of players in each state. The terminal cost depends on the player's terminal state as well as the terminal distribution of players among states.

Let $I_d=\{1,2,3,...,d\}$. The running cost of the reference player whose state is $i$ is given by a cost $c:I_d\times \mathcal{S}^d\times (\re^+_0)^d\to \re$, $c(i,\theta,\alpha)$,
where $\theta\in \mathcal{S}^d$ is the probability distribution of players among states,
and $\alpha_{j}$ is the transition rate the reference player uses to change from state $i$ to state $j$.
We suppose $c$ is
Lipschitz continuous in $\theta$,
with the Lipschitz constant (with respect to $\theta$)  bounded independently of $\alpha$.
We  suppose $c$ is differentiable with respect to $\alpha$, and that $\frac{\partial c}{\partial \alpha}(i, \theta, \alpha)$,
is Lipschitz with respect to $\theta$, uniformly in $\alpha$.

We also suppose  that $c(i,\theta,\alpha)$ does not depend on $\alpha_{i}$, is uniformly convex (on the remaining coordinates),
that is, for any $i\in I_d$, $\theta\in \Sss^d$, $\alpha,  \alpha'\in (\Rr_0^+)^d$, with $\alpha_j \neq  {\alpha'}_{j}$, for
some $j\neq i$,
\begin{equation}
\label{convc}
c(i,\theta,\alpha\,')-c(i,\theta,\alpha)\geq \nabla_{\alpha} c(i,\theta,\alpha)\cdot(\alpha\,'-\alpha)+\gamma\|\alpha\,'-\alpha\|^2.
\end{equation}
We suppose that $c$ is superlinear on $\alpha_{j}$, $j \neq i$, that is,
\[
\lim_{\alpha_j\to\infty} \frac{c(i, \theta,  \alpha)}{\|\alpha\|}\to \infty.
\]
The reference player has a terminal cost denoted by  $\psi:I_d\times \mathcal{S}^d\to \re$, $\psi^i(\theta)$. We suppose $\psi$ is
Lipschitz continuous in $\theta$, with the Lipschitz constant (with respect to $\theta$) bounded independently of $\alpha$.

\subsection{Single player control problem: the value function}

Let $T>0$ be the time duration of game.
Suppose the players are distributed among the $d$ states according to the distribution probability $\theta:[0,T]\to \mathcal{S}^d$,
which for now we assume to be known by the reference player.
Let
\[
 u_{\theta}^i(t,\alpha)=\mathbb{E}_{\bi_t=i}^{\alpha} \left[\int_t^T c(\bi_s,\theta(s),\alpha(s)) ds + \psi^{\bi_T}(\theta(T))\right]\,.
\]
We define the value function associated to $\theta$, denoted by $u_{\theta}:I_d\times[0,T]\to \re$, as
\begin{equation}\label{u}
    u_{\theta}^i(t)=\min_{\alpha} u_{\theta}^i(t,\alpha),
\end{equation}
where $\bi_s$ is a continuous time Markov chain controlled by $\alpha$ which corresponds to the state of the reference player at time $s$,
and $\mathbb{E}_{\bi_t=i}^{\alpha}$ is the expectation conditioned on the event $\bi_t=i$, given the transition rate $\alpha$.
Here the minimization is performed over Markovian controls $\alpha(s)=\alpha(\bi_s,s)$.
More precisely
\begin{equation*}
    \mathbb{P}[\bi_{s+h}=j | \bi_s] = \alpha_j(s) h + o(h)
\end{equation*}
where $\lim_{h\to 0} \frac{o(h)} h = 0$.
In \S\ref{verifMFG} existence of optimal Markovian controls will be proved. 



\subsection{Definitions and preliminary results}

Let $\Delta_i:\re^d\rar\re^d$ be the difference operator on $i$, given by
$$
\Delta_iz=(z^1-z^i,...,z^d-z^i)\,.
$$
The infinitesimal generator of a finite state continuous time Markov chain, with transition rate $\nu_{ij}$, acting on  a function $\varphi:I_d\to \re$, is given by
$$ A^{\nu}_i(\varphi)= \sum_{j} \nu_{ij} (\varphi^j-\varphi^i)= \nu_{i\cdot} \cdot \Delta_i\varphi \,.$$

We define the generalized Legendre transform of the function $c(i,\theta,\cdot)$, as
\begin{eqnarray}
 \nonumber
  h(z,\theta,i)  &=&     \min_{\mu \in (\re^+_0)^d} c(i,\theta,\mu)+ \sum_j \mu_j (z^j-z^i)
 \\
   &=&  \min_{\mu \in (\re^+_0)^d} c(i,\theta,\mu)+ \mu \cdot \Delta_iz\,. \label{LegendreTransform}
\end{eqnarray}
Because of the superlinearity and uniform convexity of $c$ the function
\begin{eqnarray}\label{defalphaestrela}
   \alpha^*(z,\theta,i) &=&  \mbox{argmin}_{\mu \in (\re^+_0)^d}  c(i,\theta,\mu)+ \mu\cdot \Delta_iz \,
\end{eqnarray}
is well defined, except for its $i-$th coordinate, since $(\Delta_iz)^i=0$.
We
will denote the $j$-th entry of the vector $\alpha^*(z,\theta,i)$ as $\alpha_j^*(z,\theta,i)$, and for convenience and definitness we set
\begin{equation}\label{somazero}
    \alpha_i^*(z,\theta,i) \equiv - \sum_{j\neq i} \alpha_j^*(z,\theta,i)\,.
\end{equation}
The definition \eqref{somazero} is consistent because $c(i,\theta,\alpha)$ does not depend on the $i$-th entry of the vector $\alpha$, and for
 that reason \eqref{defalphaestrela} does not define $\alpha_i^*(z,\theta,i)$.
The uniform convexity of $c(i,\theta,\cdot)$ shows that $\alpha^*$ is well defined. We will write $h(\Delta_iz,\theta,i)$ and $\alpha^*(\Delta_iz,\theta,i)$ to stress the fact that $h$ and $\alpha^*$ depend only on $\Delta_iz$.
Because \[h(\Delta_iz,\theta,i)=h(z,\theta,i)\] there is no ambiguity of this notation.


%



The following Proposition is proved in the Appendix.
\begin{proposition}\label{Lipschitz0}
We have
\begin{itemize}
\item[a)]
If $h$ is differentiable, for $j\neq i$
\[
\alpha_j^*(\Delta_iz,\theta,i) =
\frac{\partial h(\Delta_iz,\theta,i) }{\partial z^j},
\]
furthermore, in general,
for all $z$ and $v$
\begin{equation}\label{cvx}
h(z+v,\theta,i)-h(z,\theta,i)\leq \sum_j \alpha^*_j(z,\theta,i)\,v^j,
\end{equation}
i.e. $\alpha^*_j(z,\theta,i)\in \partial^+_z h(z,\theta,i)$, where $\partial^+$ denotes the superdiferential.
\item[b)]
The function $\alpha^*$ is Lipschitz in $p$ and in $\theta$. The Lipschitz constants are uniform. More precisely,
$$\|\alpha^*(p\,',\theta,i)-\alpha^*(p,\theta,i)\|\leq \frac{1}{\gamma}\big\|p\,'-p\,\big\|\,\,\forall \,\, p, p\,', \theta,i,$$
and
$$\|\alpha^*(p,\theta,i)-\alpha^*(p,\theta\,',i)\|\leq \frac{K_c}{\gamma}\big\|\theta -\theta\,'\big\|\,,\,\,\forall \,\, p, \, \theta, \theta\,',i. $$ where $\gamma$ is the constant given by \eqref{convc} and $K_c$ is the Lipschitz constant of $\nabla_{\alpha} c$.
\item[c)] The function $h$ is locally Lipschitz in $p$ and in $\theta$. The Lipschitz constants are uniform if $\Delta z$ is bounded.
\end{itemize}
\end{proposition}

\subsection{Hamilton-Jacobi equation and a Verification Theorem}\label{verifMFG}

We continue to assume that $\theta:[0,T]\to \Sss^d$ is given.
As in classical optimal control we introduce now the Hamilton-Jacobi ODE:
\begin{equation}\label{HJ}
\begin{cases}
    -\frac{d u^i}{dt} = h(\Delta_iu,\theta,i) , \\
    u^i(T)=\psi^i(\theta(T)).
\end{cases}
\end{equation}
This is a terminal value problem (TVP) consisting of a system of $d$ coupled ODE´s with a terminal condition given by $\psi$.
It turns out, as Theorem \ref{MFG:VerT} states, that the solution to this ODE is the value function. Before proving Theorem \ref{MFG:VerT}
we begin by proving a maximum principle for the equation \eqref{HJ}, which will be also used to prove existence and uniqueness.

\begin{proposition}\label{boundedness}
If $u$ is a solution to the HJ equation \eqref{HJ},   and $\displaystyle M=\max_{(i, \theta)\in I_d\times \mathcal{S}^d} |h(0, \theta, i)|.$
Then for all $0\leq t \leq T$ we have
$$
 \|u(t)\|\leq \|u(T)\|+2M (T-t)\,,
$$
where $\displaystyle \|u(t)\|=\max_{i \in I_d} \{|u^1(t)|, \hdots, |u^d(t)|\}$.
\end{proposition}
\begin{proof}
Let $u$ be a solution to \eqref{HJ}. Let $\tilde u=u+\rho (T-t)$. Then
$$
- \frac{d\tilde u^i}{dt}  =
  h\left(  \Delta_i \tilde u,\theta,i\right)+\rho\,.
$$
Let $(i,t)$ be a minimum point of $\tilde u$ on $I_d \times [0,T]$.
We  have $\tilde u^j(t)-\tilde u^i(t)\geq 0$ hence $\Delta_i \tilde u= (\tilde u^1(t)-\tilde u^i(t),...,\tilde u^d(t)-\tilde u^i(t))\geq 0$.
Therefore
$$-\frac{d\tilde u^i}{dt}(t)  = h\left( \Delta_i \tilde u, \theta, i\right)+\rho\geq h\left(0,\theta, i\right)+\rho \,,$$
because
if $\Delta_i p\geq 0$ we have
\[
h(\Delta_i p, \theta, i)\geq h(0,\theta, i),
\]
since $\alpha^*\geq 0$.
Furthermore, if we take $M<\rho< 2M$ we get
\[
-\frac{d\tilde u^i}{dt}(t)>0.
\]
This shows that the minimum of $\tilde u$ is achieved at $T$ hence
\[
u^i(t) \geq -\|u(T)\|-2M (T-t).
\]

Similarly, let
$(i,t)$ be a maximum point of $\tilde u$ on $I_d\times [0,T]$.
In this case we  have
$\Delta_i \tilde u\leq 0$.
Hence
$$-\frac{d\tilde u^i}{dt}(t)  = h\left(\Delta_i \tilde u, \theta, i\right)+\rho\leq h\left(0,\theta, i\right)+\rho \,.$$
Furthermore, if we take $-2 M < \rho <-M$ we get
\[
-\frac{d\tilde u^i}{dt}(t)<0.
\]
This shows that the maximum of $\tilde u$ is achieved at $T$ hence
\[
u^i(t)\leq \|u(T)\|+2M (T-t).
\]
\end{proof}

As a consequence of the last Proposition (and also using that $h$ is Lipschitz), Picard Theorem allow us to state

\begin{proposition}
The terminal value problem (TVP) given by   \eqref{HJ}
has an unique solution.
\end{proposition}

Now we prove a verification Theorem:

\begin{theorem}\label{MFG:VerT} Suppose $u:I_d \times [0,T] \rar \re$ is a solution to  the Hamilton-Jacobi terminal value problem \eqref{HJ}. Then $u$ is the value function associated to the distribution $\theta$, and
$$\tilde \alpha(i,s) \equiv \alpha^*(\Delta_iu(s), \theta(s), i)$$
is an optimal Markovian control.
\end{theorem}
\begin{proof}

The main tool for proving Theorem \ref{MFG:VerT} is the Dynkin Formula (see \cite{koloko3}, for instance):
 suppose $\alpha$ is a Markovian control continuous in time.
 Define the infinitesimal generator of the process $\bi_s$ by
\begin{equation}\label{generator-mfg}  (A^{\alpha}\varphi)^i(s) = \sum_j \alpha_{ij}(s)[\varphi^j(s)-\varphi^i(s) ] \,.
 \end{equation}

  We have that,
for any function $\varphi:I_d  \times [0,+\infty) \rar \Rr$, $C^1$ in the last variable, and any $t<T$,
\begin{equation}\label{Dynkin-mfg}  \mathbb{E}^{\alpha}_{\bi_t=i} \left[\varphi^{\bi_T}(T)-\varphi^i(t)\right]
= \mathbb{E}^{\alpha}_{\bi_t=i} \left[ \int_t^T
\frac{d\varphi^{\bi_s}}{dt} (s) + (A^{\alpha}\varphi)^{\bi_s}(s) ds  \right]\,,
 \end{equation}
where the superscript $\alpha$ means that $\bi_s$ is driven by the the control $\alpha$, while the subscript $\bi_t=i$ means we are considering the expectation conditioned on $\bi_t=i$.
We call \eqref{Dynkin-mfg} the Dynkin's formula in
analogy to the Dynkin's formula in stochastic calculus.

Now to prove the Theorem we make $\varphi=u$
in \eqref{Dynkin-mfg} .
Using the terminal condition $u^i(T)=\psi^{i}\left(\theta(T)\right)$ we have
that, for any control $\alpha$,
\begin{equation}\label{termD-mfg}
  \mathbb{E}^{\alpha}_{\bi_t=i} \left[\psi^{\bi_T}(\theta(T))-u^i(t)\right]
 = \mathbb{E}^{\alpha}_{\bi_t=i} \left[ \int_t^T
\frac{d u^{\bi_s} }{dt}(s)  + (A^{\alpha}u)^{\bi_s}(s)  ds  \right]\,.
\end{equation}
Now let $\alpha$ be any control.
In the next steps we will use the definition of $u_\theta^i(t,\alpha)$, given in \eqref{u}, and then \eqref{termD-mfg}, \eqref{generator-mfg}, and \eqref{LegendreTransform} to have
\begin{eqnarray}
  \nonumber
  u_{\theta}^i(t,\alpha) &=&\mathbb{E}^{\alpha}_{\bi_t=i} \left[\psi^{\bi_T}(\theta(T))+\int_t^T c(\bi_s,\theta(s),\alpha(s))ds\right]
\\ &=&
\nonumber
u^i(t) + \mathbb{E}^{\alpha}_{\bi_t=i} \left[ \int_t^T
\frac{d u ^{\bi_s}}{dt}(s)  + (A^{\alpha}u)^{\bi_s}(s) + c(\bi_s,\theta(s),\alpha(s))ds \right]\,
 \\ & \geq &
 \nonumber
    u^i(t) + \mathbb{E}^{\alpha}_{\bi_t=i} \left[ \int_t^T
\frac{d u^{\bi_s} }{dt}(s)  +
\min_{\mu\in (\re^+_0)^d} \sum_j \mu_j [u^j(s)-u^{\bi_s}(s)]+c(\bi_s,\theta(s),\mu) \right]
\\    &= &
\nonumber
u^i(t) + \mathbb{E}^{\alpha}_{\bi_t=i} \left[ \int_t^T
\frac{d u ^{\bi_s}}{dt}(s) +
h(\Delta_i u(s), \theta(s),\bi_s)
ds \right]
\\    &= &
\nonumber
u^i(t)\,,
\end{eqnarray}
where the last equation holds because $u$ is a solution to the  Hamilton-Jacobi equation \eqref{HJ}.
Note that in the particular case where $\alpha$ is given by  the specific control $\tilde \alpha(i,s) = \alpha^*(\Delta_i u_s,\theta_s,i)$, we have equality in the all the steps above, and therefore
we have $
  u_{\theta}^i(t,\tilde \alpha)=u^i(t)$
which show us that $\tilde \alpha$ is the optimal control and that the objective function  $u_{\theta}^i(t)$ is indeed given by $u^i(t)$.
\end{proof}

\subsection{Mean field Nash equilibria}
\label{nash_eq}
The mean field Nash equilibrium occurs when the background players are using a strategy $\beta$ for which the best response
of the reference player is $\beta$ itself, more precisely when the transition rate from $j$ to $i$ at time $s$ is given by 
\[
\beta_{ji}(s)=\alpha_{i}^*(\Delta_ju(s),\theta(s),j).
\]
The Nash equilibrium is then characterized by the system of Kolmogorov and Hamilton-Jacobi equations
\begin{equation}
\label{PVIT}
 \begin{cases}
   \frac{d}{dt}\theta^{i}  = \sum_j \theta^j \alpha_{i}^*(\Delta_ju,\theta,j) 
\\
  -\frac{d}{dt}u^i = h(\Delta_iu,\theta,i),
\end{cases}
\end{equation}
together with the initial-terminal conditions
\begin{equation}
\label{PVITDATA}
 \theta(0)=\theta_0 \qquad u^i(T)=\psi^i(\theta(T)).
\end{equation}

Note that from the ODE point of view this problem is somewhat non-standard as some of the
variables have initial conditions whereas other variables have prescribed terminal data. We call
this problem the initial-terminal value problem (ITVP) for the mean field game, and a solution of such ITVP is what we call a solution to the MFG given by $T, \theta_0,c, \psi$.

\bigskip

\subsection{Existence of Nash Equilibria in the MFG}
\label{existmfg}


We now address the existence of solutions to \eqref{PVIT} satisfying the initial-terminal conditions \eqref{PVITDATA}.
The proof of existence will be based upon a fixed point argument.

\begin{proposition}
There exists a solution to \eqref{PVIT} satisfying the initial-terminal conditions \eqref{PVITDATA}.
\end{proposition}
\begin{proof}

Let $\mathcal{F}$ be the set of continuous functions defined on $[0,T]$ and taking values in $\Sss^d$, with the $C^0$ norm.
Consider the function $\xi:\mathcal{F}\rar \mathcal{F}$ that is obtained in the following way:
given $\theta \in \mathcal{F}$, let $u_{\theta}$ be the solution of terminal value problem given by the Hamilton-Jacobi equations \eqref{HJ}
\begin{equation}
\begin{cases}
    -\frac{du^i}{dt} = h(\Delta_iu,\theta,i),  \\
    u^i(T)=\psi^i(\theta(T)).
\end{cases}
\end{equation}
We know $u_{\theta}$ depends continuously on the {\it parameters} $\theta$.

Now get the optimal control $\beta^{\theta}$ given by the Verification Theorem (Theorem \ref{MFG:VerT}):
\begin{equation}\label{optimalcontrolagain} \beta^{\theta} (i,t)=
  \mbox{argmin}_{\mu \in (\re^+_0)^d}  c(i,\theta,\mu)+ \mu\cdot \Delta_iu_{\theta} =\alpha^*(\Delta_i u_{\theta},\theta,i) \,
\,.\end{equation}

We use Proposition \ref{Lipschitz0} to conclude that $\beta^{\theta}$ is a continuous function of $u_{\theta}$, and therefore of $\theta$.

Finally, then let $\xi(\theta)$ be the solution to the  Kolmogorov equation \eqref{kolmogorov} given by the initial value problem:
\begin{equation}\label{kolmogorovagain}
\frac{d\theta^{i}}{dt}  = \sum_j \theta^j \beta_{ji}^{\theta}\,\,\,\,;\,\,\, \theta(0)= \theta_0\,.
\end{equation}

Such solution  $\xi(\theta)$ depends continuously on the parameters $\beta^{\theta}$, and therefore on $\theta$.

Therefore, using standard ODE theory we just proved that  $\xi$ is a continuous function from $\mathcal{F}$ to $\mathcal{F}$.

Now, using Proposition \ref{boundedness}, we see from \eqref{optimalcontrolagain} that  $\beta$ is bounded, with bounds that do not depend on $\theta$, and therefore from \eqref{kolmogorovagain} we have that $\xi(\theta)$ is Lipschitz, with Lipschitz  constant $\Lambda$ independent of $\theta$.

Now consider the set $\mathcal{C}$ of all Lipschitz continuous function in $\mathcal{F}$ with Lipschitz constant bounded
by $\Lambda$. This is a set of uniformly bounded and equicontinuous functions. Thus, by Arzela-Ascoli, it is a relatively compact set. It is also clear that it is a convex set.
Hence, by Brouwer fixed point Theorem, $\xi$ has a fixed point in $\mathcal{C}$.
\end{proof}

\subsection{The monotonicity hypothesis}

In order to prove the uniqueness of the MFG (\S \ref{uniq_MFG}), and also consider the convergence of solutions of MFG to stationary solutions (when $T \to \infty$ - see \S \ref{stat_uniq_conv} )
we need to introduce several monotonicity hypothesis as in the original works by Lions and Lasry.
We start with a definition:
\begin{definition}
Let  $v\in \re^d$, and set $\mathbf{1}=(1,...,1)\in\re^d$.  In $\re^d/\re$ we define the norm
$$\|v \|_{\sharp}=\inf_{\lambda\in \re}\|v +\lambda\mathbf{1} \|. $$
\end{definition}
Observe that $$\Delta_i u = \Delta_iv \;\;\forall \,1\leq i \leq d \quad\Leftrightarrow\quad \exists\; c\in \re \; \mbox{ such that }\; u=v+c\mathbf{1}
\qquad\Leftrightarrow\quad \|u-v\|_{\sharp}=0 \,.$$
Furthermore we have $$\|u\|_{\sharp}= \frac{\max_{i} (u^i)-\min_i (u^i)}{2}.$$

\begin{assumption}\label{a0} We suppose the following monotonicity hypothesis on $\psi$:
\begin{equation}\label{monotpsi}
\sum_{i}(\theta^i-\tilde\theta^i)(\psi^i(\theta)-\psi^i(\tilde\theta))\geq 0
\end{equation}
\end{assumption}

The previous assumption holds, for instance, if $\psi$ is the gradient of a convex function.

\begin{assumption}
\label{a1} We suppose that for every $M$, on the set $\|z\|_\sharp\leq M$
the function $\Delta_i z \to h(\Delta_i z)$  is uniformly concave in the non-degenerate directions, i.e., there exists $\gamma_i>0$ such that
\begin{equation}\label{concavity}
    h(\Delta_i z,\theta,i)-h(\Delta_i w,\theta,i)-\alpha^*(\Delta_i w,\theta,i)\cdot(\Delta_i z-\Delta_i w)\leq -\gamma_i\|\Delta_i z-\Delta_i w\|^2.
\end{equation}
\end{assumption}

\begin{assumption}
\label{a2} We also suppose that $h$ satisfies the following monotonicity property:
\begin{equation}\label{monot}
    \theta \cdot (h(z,\tilde \theta)-h(z,\theta))+ \tilde \theta\cdot(h(\tilde z,\theta)-h(\tilde z,\tilde \theta)   )\leq -\gamma \|\theta -\tilde \theta  \|^2,
\end{equation}
where $h(z,\theta):=(h(\Delta_1z,\theta,1),...,h(\Delta_dz,\theta,d)  ), $ and $\gamma>0$.
\end{assumption}


The last three hypothesis will be satisfied if $h$ can be written as
\[
h(\Delta_iz, \theta, i)=\tilde h(\Delta_i z,i)+f^i(\theta),
\]
with $\tilde h$ (locally) uniformly concave in the sense of \eqref{concavity} and $f$ satisfying the monotonicity hypothesis
\[
(f(\tilde \theta)-f(\theta))\cdot (\theta-\tilde \theta) \leq -\gamma |\theta-\tilde \theta|^2.
\]
The previous property holds, for instance, if $f$ is the gradient of a convex function $f(\theta)=\nabla \Phi(\theta)$.

\subsection{A key estimate}

The monotonicity hypothesis from the previous section can be used to establish both uniqueness of equilibrium solutions,
\S \ref{uniq_MFG},
and a trend to equilibrium type result \S\ref{stat_uniq_conv}. For convenience, rather than considering
the initial terminal value problem with initial values for $\theta$ at $t=0$ we consider the problem with initial values
at $t=-T$. This will be convenient when studying the trend to equilibrium, which corresponds to send $T\to \infty$
and analyzing the behavior of $(\theta_0, u_0)$.

\begin{lemma} Fix $T>0$ and suppose that $(\theta,u)$ and $(\tilde \theta,\tilde u)$ are solutions of \eqref{PVIT} with initial-terminal conditions $\theta(-T)=\theta_{-T}, u^i(T)=\psi^i(\theta(T))$  and  $ \tilde \theta(-T)=\tilde \theta_{-T} , \tilde u^i(T)=\psi^i(\tilde \theta(T))$.
Assume further that $\|u\|_\sharp, \|\tilde u\|_\sharp\leq C$.
Then
there exists a constant $C$ independent of $T$ such that, for all $0<M<T$, we have
\begin{align*}
&\int_{-M}^M  \|(\theta -\tilde \theta ) (s)\|^2+ \| (u-\tilde u)(s)\|_{\sharp}^2 ds\\&\quad \leq C\bigg(\|(\theta-\tilde \theta)(M)\|^2+\|(u-\tilde u)(M)\|_{\sharp}^2+\|(\theta-\tilde \theta)(-M)\|^2+\|(u-\tilde u)(-M)\|_{\sharp}^2\bigg).
\end{align*}
\end{lemma}
\begin{proof}
Observe that
\begin{align*}
&\frac{d}{dt}\bigg[(\theta-\tilde \theta)\cdot (u-\tilde u)\bigg]= \sum_{i=1}^d\bigg[(\dot\theta^i-\dot{\tilde \theta}^i) (u^i-\tilde u^i)+(\theta^i-\tilde \theta^i)(\dot u^i-\dot{\tilde u}^i)\bigg]
\\
&=\sum_{i=1}^d\bigg[(u^i-\tilde u^i)\bigg(\sum_j \theta^j \alpha_{i}^*(\Delta_ju,\theta,j)-\sum_j \tilde\theta^j \alpha_{i}^*(\Delta_j\tilde u,\tilde\theta,j)  \bigg)+ (\theta^i-\tilde \theta^i)(h(\Delta_i\tilde u,\tilde\theta,i)-h(\Delta_iu,\theta,i ))\bigg].
\end{align*}
In order to use the hypothesis \eqref{concavity} and \eqref{monot} we sum and subtract some terms and we change the names of the variables in the double sums.
\begin{align*}
\frac{d}{dt}\bigg[(\theta-\tilde \theta)\cdot (u-\tilde u)\bigg]&= \sum_{i=1}^d \theta^i [h(\Delta_i \tilde u,\tilde\theta,i)-h(\Delta_i \tilde u,\theta,i )]+\tilde \theta^i[h(\Delta_i u,\theta,i )-h(\Delta_i u,\tilde\theta,i)   ]
\\&+\sum_{i=1}^d \theta^i[h(\Delta_i\tilde u,\theta,i)-h(\Delta_iu,\theta,i )]+\sum_{j=1}^d\sum_{i=1}^d \theta^i \alpha_{j}^*(\Delta_iu,\theta,i)(u^j-\tilde u^j)
\\
&+\sum_{i=1}^d\tilde \theta^i[h(\Delta_iu,\tilde \theta,i )- h(\Delta_i\tilde u,\tilde \theta,i)]+\sum_{i=j}^d\sum_{i=1}^d \tilde\theta^i \alpha_{j}^*(\Delta_i\tilde u,\tilde\theta,i)(\tilde u^j- u^j).
\end{align*}
Now using that  $\displaystyle(\tilde u^i- u^i) \sum_j \alpha_{j}^*(\Delta_iu,\theta,i)=0$ and remembering that $\displaystyle \sum_{j=1}^d  \alpha_{j}^*(\Delta_iu,\theta,i)(u^j-\tilde u^j)=\alpha^*(\Delta_iu,\theta,i)\cdot(\tilde u- u)$,
  we have
\begin{align*}
\frac{d}{dt}\bigg[(\theta-\tilde \theta)\cdot (u-\tilde u)\bigg]&=\sum_{i=1}^d \theta^i [h(\Delta_i u,\tilde\theta,i)-h(\Delta_iu,\theta,i )]+\tilde \theta^i[h(\Delta_i\tilde u,\theta,i )-h(\Delta_i\tilde u,\tilde\theta,i)   ]
\\
&+\sum_{i=1}^d \theta^i\bigg[h(\Delta_i\tilde u,\theta,i)-h(\Delta_iu,\theta,i )-  \alpha^*(\Delta_iu,\theta,i)\cdot(\Delta_i \tilde u-\Delta_i u)\bigg]
\\
&+\sum_{i=1}^d\tilde \theta^i\bigg[h(\Delta_iu,\tilde \theta,i )- h(\Delta_i\tilde u,\tilde \theta,i)-  \alpha^*(\Delta_i\tilde u,\tilde\theta,i)\cdot(\Delta_i u- \Delta_i \tilde u)\bigg].
\end{align*}
Now we can use   \eqref{concavity} and \eqref{monot} to get the following estimate
\begin{equation}\label{derivadagrande}
     \frac{d}{dt}\bigg[(\theta-\tilde \theta)\cdot (u-\tilde u)\bigg]\leq  -\gamma \|\theta -\tilde \theta  \|^2-\sum_{i=1}^d (\theta^i+ \tilde \theta^i) \gamma_i\|\Delta_i u-\Delta_i\tilde u\|^2.
\end{equation}
Integrating \eqref{derivadagrande} between $-M$ and $M$, for $0<M< T$, we obtain
\begin{equation}\nonumber
    ((\theta-\tilde \theta)\cdot (u-\tilde u))(M)-((\theta-\tilde \theta)\cdot (u-\tilde u))(-M)\leq \int_{-M}^M -\gamma \|\theta -\tilde \theta  \|^2-\sum_{i=1}^d (\theta^i+ \tilde \theta^i) \gamma_i\|\Delta_i u-\Delta_i\tilde u\|^2.
\end{equation}
Note that  $(\theta-\tilde \theta)\cdot c\mathbf{1}  =0 $. Also for each $t$ there exists   $c_t\in \re$ such that $\|(u-\tilde u)(t)+c_t\mathbf{1}\|=\|(u-\tilde u)(t)\|_{\sharp}$.  Hence
\begin{align*}&\int_{-M}^M \gamma \|\theta -\tilde \theta  \|^2+\sum_{i=1}^d (\theta^i+ \tilde \theta^i) \gamma_i\|\Delta_i u-\Delta_i\tilde u\|^2 \\
&\quad\leq ((\theta-\tilde \theta)\cdot (u-\tilde u+c_{-M}\mathbf{1}))(-M)+((\theta-\tilde \theta)\cdot (\tilde u- u+c_M\mathbf{1}))(M)\\
&\quad\leq \frac 1 2 \|(\theta-\tilde \theta)(M)\|^2+\frac 1 2 \|(u-\tilde u)(M)\|_{\sharp}^2+\frac 1 2 \|(\theta-\tilde \theta)(-M)\|^2+\frac 1 2 \|(u-\tilde u)(-M)\|_{\sharp}^2.\end{align*}
Using that $\displaystyle\|\Delta_i u-\Delta_i\tilde u\|=\|u-\tilde u -(u^i-\tilde u^i)\mathbf{1}  \|\geq   \inf_{\lambda}\|u-\tilde u +\lambda\mathbf{1}\|=\|u-\tilde u\|_{\sharp}$, we have
\begin{align*}
&\int_{-M}^M \gamma \|(\theta -\tilde \theta ) (s)\|^2+ \bar \gamma\| (u-\tilde u)(s)\|_{\sharp}^2 ds \leq \int_{-M}^M \gamma \|\theta -\tilde \theta  \|^2+\sum_{i=1}^d (\theta^i+ \tilde \theta^i) \gamma_i\|\Delta_i u-\Delta_i\tilde u\|^2\\
&\leq
\frac 1 2 \left(\|(\theta-\tilde \theta)(M)\|^2+\|(u-\tilde u)(M)\|_{\sharp}^2+\|(\theta-\tilde \theta)(-M)\|^2+\|(u-\tilde u)(-M)\|_{\sharp}^2\right).
\end{align*}
Therefore we have proved
\begin{align}\label{desig}
&\int_{-M}^M  \|(\theta -\tilde \theta ) (s)\|^2+ \| (u-\tilde u)(s)\|_{\sharp}^2 ds\\&\notag \quad\leq\frac{1}{2\tilde\gamma}\bigg(\|(\theta-\tilde \theta)(M)\|^2+\|(u-\tilde u)(M)\|_{\sharp}^2+\|(\theta-\tilde \theta)(-M)\|^2+\|(u-\tilde u)(-M)\|_{\sharp}^2\bigg).
\end{align}
\end{proof}
\begin{lemma}\label{subexp}
Fix $T>0$. Suppose that $(\theta,u)$ and $(\tilde \theta,\tilde u)$ are solutions of \eqref{PVIT} with initial-terminal conditions $\theta(-T)=\theta_0, u^i(T)=\psi^i(\theta(T))$  and  $ \tilde \theta(-T)=\tilde \theta_0, \tilde u^i(T)=\tilde \psi^i(\tilde \theta(T))$. Then
\begin{equation}\label{estim} \int_{-T}^{T}\|(\theta -\tilde \theta ) (s)\|^2+ \| (u-\tilde u)(s)\|_{\sharp}^2ds\leq KT^3+4T.\end{equation}
\end{lemma}
\begin{proof}
Note that $\|(\theta -\tilde \theta ) (s)\|\leq 2$.  Let $K_0= \|\psi - \tilde \psi\|_{C_0}$. For each $-T<s<T$, by the definition of $u$ and $\tilde u$ we have
$$u^i(s)=\min_{\alpha} \mathbb{E}_{\bi_s=i}^{\alpha} \left[\int_s^T c(\bi_t,\theta_t,\alpha_t) dt + \psi^{\bi_T}(\theta_T)\right]=
 \mathbb{E}_{\bi_s=i}^{\bar\alpha} \left[\int_s^T c(\bi_t,\theta_t,\bar\alpha_t) dt + \psi^{\bi_T}(\theta_T)\right]$$
and
$$\tilde u^i(s)\leq \mathbb{E}_{\bi_s=i}^{\bar\alpha} \left[\int_s^T c(\bi_t,\tilde \theta_t,\bar\alpha_t) dt + \tilde\psi^{\bi_T}(\tilde \theta_T)\right].$$
Hence
$$\tilde u^i(s)-u^i(s)\leq \mathbb{E}_{\bi_s=i}^{\bar\alpha} \left[\int_s^T \Big(c(\bi_t,\tilde \theta_t,\bar\alpha_t)-c(\bi_t,\theta_t,\bar\alpha_t)\Big) dt + (\tilde\psi^{\bi_T}(\tilde \theta_T)-\psi^{\bi_T}(\theta_T))\right].$$
By the Lipschitz continuity of $c$ and $\psi$ in $\theta$
(remember that the Lipschitz continuity of $c$ is uniform in $\alpha$), we have
$$
\tilde u^i(s)-u^{i}(s)\leq 2TK_1 +K_0.
$$
Changing the roles of $u$ and $\tilde u$ we get
$$\|\tilde u(s)-u(s)\|_{\sharp}^2\leq KT^2+K.$$
Thus
$$ \int_{-T}^{T}\|(\theta -\tilde \theta ) (s)\|^2+ \| (u-\tilde u)(s)\|_{\sharp}^2ds\leq K T^3+(4+2K)T.$$
\end{proof}

\subsection{Uniqueness of equilibria for the initial-terminal value problem}\label{uniq_MFG}

The first consequence of the monotonicity hypothesis is the uniqueness of equilibrium solutions for the initial-terminal
value problem, which is a simple application of Lions-Lasry monotonicity method.

\begin{theorem}\label{ex_Nahs_MFG}
Suppose the monotonicity assumptions \ref{a0}, \ref{a1} and \ref{a2} hold. Then
the system (\ref{PVIT}) and (\ref{PVITDATA}) has a unique solution $(\theta,u)$.
\end{theorem}
\begin{proof}
Suppose $(\theta,u)$ and $(\tilde \theta, \tilde u)$ are solutions of (\ref{PVIT}) and (\ref{PVITDATA}).
At the initial point $t=0$ we have that $(\theta -\tilde \theta)\cdot(u-\tilde u )=0$, because $\theta_0=\tilde \theta_0.$

Integrating \eqref{derivadagrande} between $0$ and $T$, and using the terminal conditions, we have that
$$  (\theta(T)-\tilde \theta(T))\cdot (\psi(\theta(T))-\psi(\tilde\theta(T)))\leq  \int_0^T -\gamma \|\theta -\tilde \theta  \|^2-\sum_{i=1}^d (\theta^i+ \tilde \theta^i) \gamma_i\|\Delta_i u-\Delta_i\tilde u\|^2,$$ now, by assumption \ref{a0} we get
$$  0\leq  \int_0^T -\gamma \|\theta -\tilde \theta  \|^2-\sum_{i=1}^d (\theta^i+ \tilde \theta^i) \gamma_i\|\Delta_i u-\Delta_i\tilde u\|^2,$$
which implies that $\theta(s)=\tilde\theta(s)$ for all $s\in[0,T]$. Therefore,
we have the uniqueness for $\theta$. Then, once $\theta$
is known to be  unique, we obtain by a standard ODE argument that $u=\tilde u$.
\end{proof}

\subsection{Contractive mean field games}

We now introduce a condition that allow us to establish  existence of stationary solutions as well as a-priori bounds for the initial-terminal value problem.

\begin{definition}\label{defcontractive}
Let $\displaystyle\langle u\rangle= \frac 1 d \sum_j u^j\,.$
We say that $h:\mathbb{R}^d \times \Sss^d \times I_d \rightarrow \mathbb{R}$ is contractive if
there exists $M>0$ such that, $\forall \theta$, $\forall i$, if
  $\| u\|_{\sharp}>M$, then
\begin{equation}\label{contractivemax} \left( \Delta_i u \right)^j \leq 0 \;\forall \; j \mbox{ implies }
    h(\Delta_i u,\theta,i)-\langle h(u,\theta,\cdot) \rangle < 0\,,
\end{equation}
and
\begin{equation}\label{contractivemin}
\left( \Delta_i u \right)^j \geq 0 \;\forall \; j \mbox{ implies }
    h(\Delta_i u,\theta,i)-\langle h(u,\theta,\cdot) \rangle > 0\,.
\end{equation}
\end{definition}

Conditions \eqref{contractivemax} and \eqref{contractivemin} are natural if one observes that
$$ \left( \Delta_{i_1} u \right )^j \leq 0 \,\, \forall j\;\mbox{ and } \;  \left( \Delta_{i_2} u \right )^j \geq 0 \,\, \forall j$$
implies
$$ 2 \| u \|_{\sharp} = u^{i_1} - u^{i_2}\,.$$
So, if $u$ is a smooth solution to \eqref{PVIT} and $\| u(t) \|_{\sharp}$ is differentiable with $\| u(t) \|_{\sharp}>M$
then
$$ \frac{d}{dt} \| u \|_{\sharp} >0\,,$$
which implies the flow is backwards contractive with respect to the $\| \cdot \|_{\sharp}$ norm of the $u$ component.

The contractivity condition can be verified explicitly in many examples as we will illustrate in what follows.
Consider the particular case
\begin{equation}\label{cexample}
    c(i,\theta,\alpha)=\sum_j \frac{\alpha_j^2} 2 +f^i(\theta)\,,
\end{equation}
where $f^i(\theta)$ is continuous on $\theta \in \Sss^d$.
We have in this case that
\begin{equation}\label{hexample}
    h(\Delta_i u,\theta,i)= f^i(\theta)- \frac 1 2 \sum_j [(u^i-u^j)^+]^2\,.
\end{equation}
We will show now that $h$ is contractive.
Suppose   first $\left( \Delta_i u \right)^j \leq 0$ $\forall j$.
As all other cases are similar we assume $i=1$ and
\begin{equation}\label{dexample}
   u^1 \geq u^2 \geq ... \geq u^d\,.
\end{equation}
Therefore
\begin{align*}
    h(\Delta_1 u,\theta,1)-\langle h(u,\theta,\cdot)\rangle &=\frac{d-1}{d} \left[ -\frac 1 2 \sum_{j>1} (u^1-u^j)^2\right] \\
& + \frac 1 {2d}  \Big( \sum_{j>2} (u^2-u^j)^2+\sum_{j>3} (u^3-u^j)^2+...+(u^{d-1}-u^d)^2\Big)+F_1(\theta)
\end{align*}
where $F_1(\theta)$ is a bounded function of $\theta$, namely
\begin{equation}\label{Mexamp}
    \|F_1(\theta)\| = \left\|\frac{d-1}{d} f^1(\theta) - \frac{1}{d} \sum_{j>1} f^j(\theta)\right\|
    \leq 2 \max_{\theta, i }  f^i(\theta) \,.
\end{equation}
Now we multiply by $-2d$ to have
\begin{align*}
  -2d \left(  h(\Delta_1 u,\theta,1)-\langle h(u,\theta,\cdot)\rangle \right) &=(d-1) \left[ \sum_{j>1} (u^1-u^j)^2\right]\\
& -   \Big( \sum_{j>2} (u^2-u^j)^2+\sum_{j>3} (u^3-u^j)^2+...+(u^{d-1}-u^d)^2\Big)-2d F_1(\theta)\,.\end{align*}
Reordering we have
\begin{align*}
 & -2d \left(  h(\Delta_1 u,\theta,1)-\langle h(u,\theta,\cdot)\rangle \right) \\
 &=\sum_{j>1} (u^1-u^j)^2 +   \Big( \sum_{j>1} (u^1-u^j)^2 - \sum_{j>2} (u^2-u^j)^2 \Big)\\
  &+\Big(\sum_{j>1} (u^1-u^j)^2- \sum_{j>3} (u^3-u^j)^2\Big) +...+\Big( \sum_{j>1} (u^1-u^j)^2-(u^{d-1}-u^d)^2\Big)- 2dF_1(\theta)\,.\end{align*}
Now using \eqref{dexample} we have an inequality
\begin{align*}& -2d \left(  h(\Delta_1 u,\theta,1)-\langle h(u,\theta,\cdot)\rangle \right) \\
 &\geq\sum_{j>1} (u^1-u^j)^2 +   \Big( \sum_{j>1} (u^1-u^j)^2 - \sum_{j>2} (u^1-u^j)^2 \Big)\\
  &+\Big(\sum_{j>1} (u^1-u^j)^2- \sum_{j>3} (u^1-u^j)^2\Big) +...+\Big( \sum_{j>1} (u^1-u^j)^2-(u^{1}-u^d)^2\Big) -2d F_1(\theta) \,,\end{align*}
which implies
\begin{align*} & -2d \left(  h(\Delta_1 u,\theta,1)-\langle h(u,\theta,\cdot)\rangle \right) \\
 &\geq\sum_{j>1} (u^1-u^j)^2 +   \Big( (u^1-u^2)^2 \Big)\\
  &+\Big( (u^1-u^2)^2+  (u^1-u^3)^2\Big) +...+\Big( \sum_{j=1}^{d-1} (u^1-u^j)^2\Big) -2d  F_1(\theta) \,,\end{align*}
and the last inequality implies that $  h(\Delta_1 u,\theta,1)-\langle h(u,\theta,\cdot)\rangle  < 0$ whenever
$\| u\|_{\sharp}$ is large enough.

For the case $\left( \Delta_i u \right)^j \geq 0$ $\forall j$ it suffices, as before, to assume $i=1$ and
$$ u^1\leq u^2\leq ...\leq u^d.$$
Then
\begin{align*}& h(\Delta_1 u,\theta,1)-\langle h(u,\theta,\cdot)\rangle   \\
&= \frac 1 {2d}  \Big( \sum_{j<2} (u^2-u^j)^2+\sum_{j<3} (u^3-u^j)^2+...+\sum_{j<d}(u^{d}-u^j)^2\Big)+F_1(\theta).\end{align*}
This implies
$$ h(\Delta_1 u,\theta,1)-\langle h(u,\theta,\cdot)\rangle>0\,,  $$ whenever $\| u\|_{\sharp}$ is large enough.


\subsection{Stationary solutions}

We now discuss stationary solutions to  \eqref{PVIT}. It is clear, if, for instance $h>0$ the equation \eqref{PVIT}
cannot admit stationary solutions in the sense that $\frac{d}{dt}u=\frac{d}{dt} \theta=0$.
Therefore we need to consider
 stationary solutions to \eqref{PVIT} modulo addition of a constant:
\begin{definition}
A triplet $(\bar \theta,\bar u, \kappa) \in \mathcal{S}^d\times \re^d\times \re$ is called a stationary solution of \eqref{PVIT}
if
 \begin{eqnarray}
\begin{cases}
   \sum_j \bar\theta^j \alpha_{i}^*(\Delta_j\bar u,\bar\theta,j)=0\,,\\
   h(\Delta_i\bar u,\bar\theta,i)=\kappa\,.
\end{cases}\label{equilvectors}
\end{eqnarray}
\end{definition}
If $(\bar \theta,\bar u, \kappa)$ is a stationary solution for the MFG equations, then $(\bar \theta,\bar u-\kappa t)$ solves \eqref{PVIT}.

\begin{proposition}\label{te:stat_existence}
Suppose $h:\mathbb{R}^d \times \Sss^d \times I_d \rightarrow \mathbb{R}$ given by \eqref{LegendreTransform} is contractive.
\begin{itemize}
\item[(a)] For $M$ large enough, the set $\left\{ u \in \re^d , \|u\|_{\sharp} < M \right\} \times \Sss^d$ is invariant backwards in time by the flow of equation \eqref{PVIT}.
\item[(b)] There exist a stationary solution of \eqref{PVIT}.
\end{itemize}
\end{proposition}

\begin{proof}
The first item is a direct consequence of the definition \ref{defcontractive} and the observations thereafter.
The second item is a consequence of Brower fixed point theorem for flows that leave invariant compact and convex sets.
\end{proof}

\subsection{Uniqueness of stationary solutions and trend to equilibrium}
\label{stat_uniq_conv}

We now discuss two important consequences of the monotonicity and contractivity properties: the
uniqueness of solutions  and the trend to equilibrium.

\begin{theorem}\label{te:stat_uniqueness}
Suppose that the monotonicity assumptions \ref{a1}, \ref{a2}, and contractivity hold.
\begin{itemize}
\item[(a)] Suppose $\|u(T)\|_{\sharp} \leq M$, where $u$ is a solution to \eqref{PVIT}, and $M$ is large enough. Then $\|u(t)\|_{\sharp} \leq M \,\forall\,t  \in [0,T]$.
\item[(b)] The stationary solution $(\bar \theta, \bar u,\kappa)$ is unique (up to the addition of a constant to $\bar u$).
\item[(c)] Given $T>0$, a vector $\theta_0$, and a terminal condition $\psi$, let $(\theta^T,u^T)$ be the solution of
 \eqref{PVIT} with initial-terminal conditions $\theta^T(-T)=\theta_0$
and ${u^{T,i}}(T)=\psi^i(\theta^T(T))$. We have, when $T\to\infty$
$$\theta^T(0)\to \bar\theta,\,\,\,\, \|u^T(0)-\bar u\|_{\sharp}\to 0 ,$$
where   $(\bar\theta,\bar u)$ is the unique stationary solution for the MFG equations.
\end{itemize}
\end{theorem}
\begin{proof}
Item (a) is again a a direct consequence of the definition \ref{defcontractive} and the observations thereafter. 

In order to prove items (b) and (c), fix two probability distributions $\theta_0$ and $\tilde \theta_0$ in $\mathcal{S}^d$, and two terminal conditions $\psi$ and $\tilde \psi$. For each $T>0$, let $(\theta^T,u^T)$ and $(\tilde \theta^T,\tilde u^T)$ be the solutions of \eqref{PVIT} with initial-terminal conditions $\theta^T(-T)= \theta_0, u^{T,i}(T)=\psi^i(\theta^T(T))$  and  $ \tilde \theta^T(-T)=\tilde \theta_0 , \tilde u^{T,i}(T)=\tilde \psi^i(\tilde \theta^T(T))$, respectively. By the contractivity hypothesis, $\|u(t)\|_{\sharp}$ and $\|\tilde u(t)\|_{\sharp}$ are uniformly bounded.

We define
$$ f_T(s):= \|(\theta^T -\tilde \theta^T ) (s)\|^2+ \| (u^T-\tilde u^T)(s)\|_{\sharp}^2,$$
   and, for $0<\tau<T$,
 $$F_T(\tau):= \int_{-\tau}^\tau f_T(s)ds.$$
By \eqref{desig}, we have
 $$F_T(\tau)\leq \frac{1}{\tilde\gamma}(f_T(\tau)+f_T(-\tau)).$$
Note  that $\dot F_T(\tau)=f_T(\tau)+f_T(-\tau)$, hence
$$ F_T(\tau)\leq \frac{1}{\tilde\gamma} \dot F_T(\tau).$$
This implies
 $\frac{d}{dt}\ln F_T(\tau)\geq \tilde\gamma$, therefore
 \[\ln F_T(\tau)-\ln F_T(1)\geq (\tau-1)\tilde \gamma,
 \] for all $0<\tau<T$. From this we get
 $$ \int_{-1}^{1}f_T(s)ds=F_T(1)\leq \frac{F_T(T)}{e^{(T-1)\tilde \gamma}}\to 0 \qquad \mbox{when } T\to \infty,$$
because $F$ has sub-exponential growth, by \eqref{estim} in Lemma \ref{subexp}.

Now there exists $t(T)\in[-1,1]$ with $f_{T}(t(T))\leq \frac{F_T(1)}{2}$. Hence
$$
\|\theta^{T}(t(T)) -\tilde \theta^{T}  (t(T))\| \to 0,\,\,\,  \| u^{T}(t(T))-\tilde u^{T}(t(T))\|_{\sharp} \to 0\,,
$$
as $T\to+\infty$.

Recall that $(\theta^{T}, u^{T})$ and $(\tilde\theta^{T}, \tilde u^{T})$ are solutions of the same time-homogeneous ODE \eqref{PVIT}, with data at time $t_T$  $(\theta^{T}(t_T), u^{T}(t_T))$ and $(\tilde\theta^{T}(t_T), \tilde u^{T}(t_T))$ whose difference goes to zero as as $T\to+\infty$. From the continuous dependence of solutions of ODE´s with respect to initial conditions, and observing that $t_T \in [-1,1]$, we can conclude that
 $$\|\theta^{T}(t) -\tilde \theta^{T}(t)\|\to 0\,\,\,\mbox{and} \,\,\,  \| u^{T}(t)-\tilde u^{T}(t)\|_{\sharp}\to 0\;,\quad \mbox{uniformly in } t \in[-1,1]\,, \,\mbox{as } T \to \infty.$$

Now, from Theorem \ref{te:stat_existence}, we know there exists a stationary solution $(\bar \theta, \bar u)$. If we choose initial and terminal conditions $(\tilde \theta,\tilde \psi)$
in such a way that $(\tilde \theta^T(t), \tilde u^T(t)) = (\bar \theta, \kappa (T-t) + \bar u )$, we have the convergence of $(\theta^T, u^T)$ to $(\bar \theta, \bar u)$, which implies both the trend to equilibrium and the uniqueness of stationary solutions.
\end{proof}

%
%



\section{The $N+1$-player game}
\label{nplayer}


In this section we consider games between $N+1$-players
which are symmetric under permutation of players.
As in the previous section we assume that each of
the players can be in one of $d$ states, and knows, in addition to his or her state,
the number of players in each of the states.

Players follow a Markovian dynamics in which each player controls the switching rate, as discussed in \S\ref{cmdyn}, and \S\ref{ippv}. Using Hamilton-Jacobi ODE methods, \S \ref{thjode}, and a verification Theorem, \S \ref{verifteo}, we formulate the Nash equilibrium problem for the $N+1$-player problem.
Maximum principle type estimates are considered in \S
\ref{hjode} which are then applied to establishing the existence of Nash equilibrium solutions
in \S\ref{eqsols}.

\subsection{Controlled Markov Dynamics}
\label{cmdyn}


Remember that $I_d=\{1,2,3,...,d\}$, and let $S^d_N= \{(n^1,...,n^d) \in \mathbb{Z}^d \| \sum_{i=1}^d n^i = N, n^i\geq 0\}$. Let $e_k$ be the $k-th$ vector of the canonical basis of $\re^d$, and let $e_{jk}=e_j-e_k$.

In the preceding section we considered a game where a very large number of players was allowed to switch between $d$ states.
The fraction of players in each state was approximated by a deterministic vector $\theta(t)$.
In an analogous way,
we consider now a game between $N+1$ players that are allowed to switch between the same $d$ states. As before, to describe the game we will use a reference player. 
However, we no longer make the assumption that the fraction of the players in each state can be approximated
by a deterministic vector $\theta(t)$.
Instead, in addition to the position $\bi_t$ of the reference player,  we consider a second controlled Markov Chain $\bn$,
taking values in $S^d_N$, which records the number of the remaining players (distinct from the reference player) that are in any of the $d$ states at any given time.
Each player knows his own state, as well as the number of remaining players that are in any of the  states. No further information is available to any individual player.

We suppose the reference player
switches from state $i$ to state $j$ according to a switching Markovian rate  $\alpha_{ij}(n, t)$ which he or she would like to optimize upon.
We suppose that each of the players distinct from the reference player follows a controlled Markov process $\bk_t$ with transition rates from state $k$ to
state $j$ given by  $\beta=\beta_{kj}(n,t)$.
More precisely, we have, for $j\neq k$,
$$
\mathbb{P}\Big(\bk_{t+h}=j\|\bn_t=n,\bk_t=k\Big)= \beta_{kj}(n,t).h+o(h)\,,
$$
where $\lim \frac{o(h)}h =0$ when $ h \rar 0$.
We suppose that $\beta:I_d\times I_d \times S^d_N \times[0,+\infty) \rar \re$ is an {\it admissible control},
that  it is bounded and continuous as a function of time, and $\beta_{kk}(n,t)= - \sum_{j \neq k} \beta_{kj}(n,t) \;\forall \;k,n,t$ and $\beta_{kj}(n,t)\geq 0 \;\forall \; k \neq j$.
We assume further that the state transitions of the different players
are independent, conditioned on $\bi$ and $\bn$.


From the symmetry and independence of transitions assumption, for $k \neq j$, we have
\begin{align*}
\mathbb{P}\Big(\bn_{t+h}=n+e_{jk}\|\bn_t&=n,\bi_t=i\Big)= \gamma_{\beta,kj}^{n,i}(t).h+o(h)\,,
\end{align*}
where $\lim \frac{o(h)}h =0$ when $ h \rar 0$ and the transition rates of the process $\bn$ are given by
\begin{align}\label{gamma}
\gamma_{\beta,kj}^{n,i}(t)&= n_k \beta_{kj}(n+e_{ik},t).
\end{align}
The previous expression for the rate, namely the term $n+e_{ik}$ instead of $n$, follows from the fact that from the point of view of a player which is in state $k$, and is distinct from the reference player,  the number of other players in any state is given by $\bn+e_{\bi}-e_k = \bn+e_{\bi k}$.
Note that the rate function $\beta$ is a deterministic time-dependent function, which makes $(\bn,\bi)$ a
non-time homogeneous Markov process.

\subsection{Individual player point of view}
\label{ippv}
\label{acdf}




The reference player would like to choose its transition rate $\alpha$, possibly different from $\beta$,
in order to minimize
\begin{equation}\label{defu}
u^i_n(t,\beta,\alpha)= \mathbb{E}^{\beta,\alpha}_{A_t(i, n)} \left[ \int_t^T c\left(\bi_s,\frac{\bn_s}N,\alpha(s)\right)ds + \psi^{\bi_T}\left(
\frac{\bn_T}N\right) \right]\;,
\end{equation}
 where the subscript $A_t(i, n)$ means we are considering the expectation conditioned on $\bi_t=i, \bn_t=n$.
That is, reference player looks for the control $\alpha$ which is a solution to the minimization problem
$$
u^i_n(t;\beta)= \inf_{\alpha} u^i_n(t,\beta,\alpha),
$$
where the minimization is performed over the set of all admissible controls $\alpha$.
We will call the function $u^i_n(t;\beta)$ above
the value function for the reference player  associated to
the strategy $\beta$ of the remaining $N$ players.
The control $\alpha$ that attains the minimum above can be called the best response (for the reference player) to a control $\beta$. 



\medskip

\subsection{The Hamilton-Jacobi ODE for the N+1-player game}\label{hjode}
\label{thjode}

Fix an admissible control $\beta$.
Consider the system of ODE´s indexed by $i$ and $n$ given by
\begin{equation}\label{HJ_prim}
- \frac{d\varphi^i_n}{dt} (t) =
 \sum_{k,j}\gamma_{\beta,kj}^{n,i} (t) \big(\varphi^i_{n+e_{jk}}(t)-\varphi^i_n(t)\big)
 + h\left(\Delta_i\varphi_n (t),  \frac n N, i\right) \,,
 \end{equation}
where
$\gamma_\beta$ is given by \eqref{gamma}, and, as before,
 $\Delta_i\varphi_n (t) = \big(\varphi_n^1(t)-\varphi_n^i(t), \hdots , \varphi_n^d(t)-\varphi_n^i(t)\big)$.

This system of ODE is called the Hamilton-Jacobi (HJ) ODE for the $N+1$-player game associated to the strategy  $\beta$ of the remaining $N$ players.

We denote by
\begin{equation}\label{NormaInf}
\|u(t)\|_{\infty}=\max_{n,i} |u_n^i( t)|,
\end{equation}

The proof of the next Proposition is analogous to the proof of Proposition \ref{boundedness} and it is postponed to the Appendix.
\begin{proposition}\label{maxpri}
Let $u$ be a solution to \eqref{HJ_prim} and $\displaystyle M=\max_{(i, \theta)\in I_d\times \mathcal{S}^d} |h(0, \theta, i)|.$
Then for all $0\leq t \leq T$ we have
\[
 \|u(t)\|_{\infty}\leq \|u(T)\|_{\infty}+2M (T-t).
 \]

\end{proposition}


As a consequence of $h$ 
being locally Lipschitz continuous, Picard Theorem, together with the previous bound, allow us to establish
\begin{theorem}\label{existunic_HJ_Npl}
The terminal value problem (TVP) given by equation \eqref{HJ_prim} and the terminal condition $\varphi^i_n(T)=\psi^i( \frac n N)$ has a unique solution.
\end{theorem}

\subsection{A verification Theorem for the N+1-player game}\label{verifteo}

Now we state  a verification Theorem, which is completely analogous to the respective verification Theorem of the preceding section:
The corresponding proof can be found in the Appendix.
\begin{theorem}\label{verif-Npl}
Let $v$ be a solution to \eqref{HJ_prim} satisfying the
terminal condition $v^i_n(T)=\psi^i\left(\frac{n}N\right)$.
Then
$$u^i_n(t;\beta)=v^i_n(t)\,.$$
Also, the Markovian control
\begin{equation}\label{control}
\tilde \alpha(\beta)(i,n,t)\equiv \alpha^* \left(\Delta_i v_n( t), \frac n N, i\right),
\end{equation}
is admissible 
and satisfies
$$u^i_n(t;\beta)= u^i_n(t,\beta,\tilde \alpha(\beta))\,.$$
\end{theorem}

\medskip
Thus a classical solution to the HJ equation associated to $\beta$ is the value function corresponding to $\beta$ and determines an optimal admissible control $\tilde \alpha(\beta)$, for the reference player.




\subsection{Equilibrium solutions}
\label{eqsols}


We now consider Nash equilibria for the $N+1$-player game. For that we look
for controls $\beta$ for which the best response of any player to $\beta$ is $\beta$ itself.

\medskip

\begin{definition}
An admissible control $\beta$ is a Nash equilibrium if  $\tilde\alpha(\beta)=\beta$.
\end{definition}

\begin{theorem}\label{nash_ex_uniq}
There exists a unique Nash equilibrium $\bar\beta$.
\end{theorem}
\begin{proof}
A necessary condition for a control $\bar\beta$ to be a Nash equilibrium is that
from \eqref{control}, we have
\[
{\bar\beta}_{kj}(n,t)=\alpha^*_j\left(\Delta_k u_n( t;\bar\beta), \frac n N, k\right).
\]
Hence this gives rise to the system of  nonlinear differential equations
\begin{equation}
\label{eqhj}
- \frac{du_n^i}{dt}  =
 \sum_{k,j} \gamma^{n,i}_{ kj} (u^i_{n+e_{jk}} - u^i_n  )  + h\left(\Delta_i u_n, \frac n N, i\right)\,,
\end{equation}
with terminal condition
\begin{equation}\label{term-npl}
    u^i_n(T)=\psi^i\left(\frac{n}N\right)\,\,\,\,\forall\,\, i\in I_d, n\in\Sss^d_N ,
\end{equation}
 where $\gamma^{n,i}_{ kj}$ are given by
\begin{align}\label{gammaeq}
\gamma^{n,i}_{ kj}&= n_k \alpha_j^*\Bigg(\Delta_k u_{n+e_{ik}},\frac{n+e_{ik}}{N},k\Bigg)\,.
\end{align}
Note that \eqref{eqhj} is well posed because $u_n$ is bounded and the right-hand side is Lipschitz and admits a unique solution. Hence
existence and uniqueness of a Nash equilibrium follows.
\end{proof}

The following property of  $\gamma^{n,i}_{ kj}$ will be proved in the Appendix:
\begin{lemma} \label{estimgamma}Let us suppose that $\|\Delta_k u_n\|_{\infty}$ is bounded, and denote by $z^l_{n,sr}= u^l_{n+e_{rs}}-u^l_n. $ Then we have
\begin{equation}
\Big|\gamma_{kj}^{n+e_{rs},i}-\gamma_{kj}^{n,i}\Big|\leq C+CN\max_{rs}\|z^{\cdot}_{\cdot, sr} \|_{\infty}.
\end{equation}
\end{lemma}

\section{Convergence}
\label{convsec}

This last section addresses the convergence as the number of players tends to infinity to the mean field model
derived in the previous section.

We start this section by discussing some preliminary estimates in \S\ref{pest}. Then, in
\S \ref{uest} we establish uniform estimates for $|u_{n+e_{rs}}-u_n|$, which are essential
to prove our main result, Theorem \ref{teoconv}, which is discussed in \S\ref{convsubsec}.
This theorem shows that the model derived in the previous section can be obtained as an appropriate limit
of the model with $N+1$ players discussed in section \ref{nplayer}.

\subsection{Preliminary results}
\label{pest}
Let us denote by $m=(i,n)\in I_d\times\Sdn$, and
consider the system of ordinary differential equations
\begin{align*}
-\dot z^i_n & = \sum_{k,j} a^{i}_{n,kj} \bigg(z^i_{n+e_{jk}}-z^i_n\bigg) + \sum_l a_{n}^{l,i} (z^l_n-z^i_n),
\end{align*}
where $a^{i}_{n,kj} \geq 0$ and $a_{n}^{l,i} \geq 0$.
Note that this system is a particular case of 
\begin{equation}
\label{unpert}
-\dot z_m=\sum_{m'\in I_d\times\Sdn}a_{mm'}(t) (z_{m'}-z_m),
\end{equation}
where $a_{mm'}(t)\geq 0$.
We write \eqref{unpert} in compact form as
\begin{equation}
\label{unpert2}
-\dot z(t)=M(t)z(t).
\end{equation}
The solution to this equation with terminal data $z(T)$ can be written as
\begin{equation}\label{unpert3}
z(t)=K(t,T)z(T),
\end{equation}
where $K(t, T)$ is the fundamental solution to \eqref{unpert2} with $K(T,T)=I$. Note that equations \eqref{unpert2} and \eqref{unpert3} imply
\begin{equation}\label{unpert4}\frac{d}{dt}K(t,T)=-M(t)K(t,T). \end{equation}

The proofs of Lemma \ref{contprop}, Lemma \ref{413} and Lemma \ref{tfc}  can be found in Appendix.

\begin{lemma}
\label{contprop}
For $t<T$ we have
\[
\|z(t)\|_{\infty}\leq \|z(T)\|_{\infty}
\]
(see \eqref{NormaInf}).
Furthermore, if $z(T)\leq 0$ then $z(t)\leq 0$.
\end{lemma}

From the previous Lemma we also conclude
\begin{lemma}\label{ordpr}
If $p_1\leq p_2$, and $t\leq s$, then we have
\[
K(t,s)p_1\leq K(t, s)p_2,
\]
which means $K(t,s)$ is an order preserving operator.
\end{lemma}
\begin{proof}
Observe that if $p_1-p_2\leq 0$ then $K(t, s)(p_1-p_2)\leq 0$, by Lemma \ref{contprop}.
\end{proof}

\begin{lemma}\label{413}
Suppose $z$ is a solution to
\begin{equation}
\label{difeneq}
-\dot z(s)\leq M(s) z(s) +f(z(s)).
\end{equation}
where $M(t)$ was defined in \eqref{unpert} and \eqref{unpert2}.
Then, for all $m=(i,n)\in I_d\times \Sdn$
\[
z^i_n(t)=z_m(t)\leq \|z(T)\|_{\infty}+\int_t^T \|f(z(s))\|_{\infty}ds.
\]
\end{lemma}

\begin{lemma}\label{tfc}
Suppose $v:[0,T]\to\re$ is a solution to the ODE with terminal condition
\begin{equation}\label{tfc_p}
\begin{cases}
-\frac{dv}{ds}= Cv+C N v^2 + \frac{C}{N},\\
v(T)\leq\frac{C}{N},
\end{cases}
\end{equation}
where $N$ is a natural number, and $C>0$. Then, there exists $T^{\star}>0$, which does not depend on $N$, such that
$T \leq T^{\star}$ implies
$v(s) \leq \frac{2C}{N}$ for all $0\leq s \leq T$.
\end{lemma}

%

\subsection{Gradient estimates}\label{uest}

In this section we prove "gradient estimates" for the $N+1$-player game, that is,
we assume that the difference $u_{n+e_{rs}}-u_n$ is of the order $\frac 1 N$ at time $T$
and show that it remains so for $0\leq t\leq T$, as long as $T$ is sufficiently small.

\begin{proposition}
\label{lipbounds}
Let $u^i_n(t)$ be a solution of \eqref{eqhj} with terminal conditions \eqref{term-npl}. Then there exists  $C>0$  and $T^{\star}>0$ such that, for   $0<T< T^{\star}$, we have
\[
\max_{rs}\|u^i_{n+e_{rs}}(t)-u^i_n(t)\|_{\infty}\leq \frac {2C} N,
\] for all $0\leq t\leq T$.
\end{proposition}

Before proving the Proposition, we remember the norm $\| \cdot \|_{\infty}$ was defined in \eqref{NormaInf}.

\begin{proof}
Using the terminal condition \eqref{term-npl} and remembering that $\psi$ is Lipschitz continuous, we know that there is a constant $C>0$ such that
\begin{equation}\label{lbh}
\max_{rs}\|u^i_{n+e_{rs}}(T)-u^i_n(T)\|_{\infty}\leq \frac C N.
\end{equation}

Let $z^{i}_{n,sr}=u^i_{n+e_{rs}}-u^i_n$. 
We have
\begin{align*}
-\dot z_{n,sr}^{i}=&\sum_{k,j}\left[\gamma_{kj}^{n+e_{rs},i}\bigg(u^i_{n+e_{rs}+e_{jk}}-u^i_{n+e_{rs}}\bigg)-\gamma_{kj}^{n,i }\bigg(u^i_{n+e_{jk}}-u^i_{n}   \bigg)           \right]\\
+& h\bigg(\Delta_i u_{n+e_{rs}},\frac{n+e_{rs}}{N},i\bigg)-h\bigg(\Delta_i u_{n},\frac{n}{N},i\bigg)\\
=&\sum_{k,j}\left[\gamma_{kj}^{n+e_{rs},i}z^{i}_{n+e_{rs},kj}-\gamma_{kj}^{n,i}\,z^{i}_{n,kj}           \right]
+ h\bigg(\Delta_i u_{n+e_{rs}},\frac{n+e_{rs}}{N},i\bigg)-h\bigg(\Delta_i u_{n},\frac{n}{N},i\bigg)\\
=&\sum_{k,j}\left[ \bigg(\frac{\gamma_{kj}^{n+e_{rs},i}+\gamma_{kj}^{n,i}}{2}\bigg) \bigg(z^{i}_{n+e_{rs},kj}-z^{i}_{n,kj}\bigg)\right]+\sum_{k,j}\left[ \bigg(\frac{\gamma_{kj}^{n+e_{rs},i}-\gamma_{kj}^{n,i}}{2}\bigg) \bigg(z^{i}_{n+e_{rs},kj}+z^{i}_{n,kj}\bigg)\right]\\
+&h\bigg(\Delta_i u_{n+e_{rs}},\frac{n+e_{rs}}{N},i\bigg)-h\bigg(\Delta_i u_{n},\frac{n}{N},i\bigg).
\end{align*}
Note that $z^{i}_{n+e_{rs},kj}-z^{i}_{n,kj}=u^i_{n+e_{rs}+e_{jk}}-u^i_{n+e_{rs}}-u^i_{n+e_{jk}} +u^i_n =z^{i}_{n+e_{jk},sr}-z^{i}_{n,sr}$.

From Lemma \ref{estimgamma}, we have
$\displaystyle\left|\bigg(\frac{\gamma_{kj}^{n+e_{rs},i}-\gamma_{kj}^{n ,i}}{2}\bigg)\right|\leq C+CN\max_{rs}\|z_{\cdot,sr}^{\cdot}\|_{\infty}.$
And note that
\[
\displaystyle \sum_{k,j} \bigg(z^{i}_{n+e_{rs},kj}+z^{i}_{n,kj}\bigg)\leq 2\sum_{k,j} \|z^\cdot_{\cdot,kj} \|_{\infty}\leq 2d^2 \max_{k,j}\|z^\cdot_{\cdot,kj} \|_{\infty}.
\]
 Hence
$$\sum_{k,j}\left[ \bigg(\frac{\gamma_{kj}^{n+e_{rs},i}-\gamma_{kj}^{n ,i}}{2}\bigg) \bigg(z^{i}_{n+e_{rs},kj}+z^{i}_{n,kj}\bigg)\right]\leq C\max_{rs}\|z^\cdot_{\cdot,kj}\|_{\infty}+CN\max_{rs}\|z^\cdot_{\cdot,kj}\|_{\infty}^2. $$
Using  item (a) of Proposition \ref{Lipschitz0} and also that $\displaystyle z^i_{n,sr}\sum_l \alpha^*_l\Big(u_n,\frac n N ,i\Big)=0$, we have
\begin{align*}
h\bigg(\Delta_i u_{n+e_{rs}},\frac{n+e_{rs}}{N},i\bigg)-h\bigg(\Delta_i u_{n},\frac{n}{N},i\bigg)&=h\bigg(\Delta_i u_{n+e_{rs}},\frac{n+e_{rs}}{N},i\bigg)-h\bigg(\Delta_i u_{n+e_{rs}},\frac{n}{N},i\bigg)\\
&+h\bigg(\Delta_i u_{n+e_{rs}},\frac{n}{N},i\bigg)-h\bigg(\Delta_i u_{n},\frac{n}{N},i\bigg)\\
& \leq \frac{C}N+ \sum_l\alpha^*_l\left(\Delta_i u_n,\frac n N ,i\right)\big(z^{l}_{n,sr}-z^{i}_{n,sr}\big).
\end{align*}
Now  denoting by $a_{n,kj,sr}^{i}=\frac{\gamma_{kj}^{n+e_{rs},i}+\gamma_{kj}^{n,i}}{2}$, and $a_{n}^{l,i}=\alpha^*_l\left(\Delta_i u_n,\frac n N ,i\right)$, we get
\begin{align*}
-\dot z^{i}_{n,sr} & \leq  \sum_{k,j}\left[ a_{n,kj,sr}^{i} \bigg(z^{i}_{n+e_{jk},sr}-z^{i}_{n,sr}\bigg)\right] + \sum_l a_{n}^{l,i} (z^{l}_{n,sr}-z^{i}_{n,sr})
+f(z),
\end{align*}
where $\displaystyle f(z)=\frac{C}{N}+C\max_{rs}\|z^\cdot_{\cdot,sr}\|_{\infty}+CN\max_{rs}\|z^\cdot_{\cdot,sr}\|_{\infty}^2$.

At this point we are in position to apply Lemma \ref{413} from the previous section. We obtain
$$z^{i}_{n,sr}(t)\leq \|z^\cdot_{\cdot,sr}(T)\|_{\infty} +\int_t^T C\max_{rs}\|z^\cdot_{\cdot,sr}(s)\|_{\infty}+CN\max_{rs}\|z^\cdot_{\cdot,sr}(s)\|_{\infty}^2+\frac {C} N  \;ds\;. $$
Finally, as $z^{i}_{n,sr}=u^i_{n+e_{rs}}-u^i_{n}$, if we set $\displaystyle w=\max_{rs}\|u^i_{n+e_{rs}}-u^i_n\|_{\infty}$ we conclude that
\[
w(t)\leq w(T)+\int_t^T Cw(s)+CN w(s)^2+\frac{C}{N} ds.
\]
Now we define
$$\eta(t) = w(T)+\int_t^T Cw(s)+CN w(s)^2+\frac{C}{N} ds.$$
We have that
\begin{equation}\label{blabla}
w(t)\leq \eta(t),
\end{equation}
 and also that $$\frac{d \eta}{dt}(t) = - g(w(t)),$$ where $g$ is the nondecreasing function $g(w)= Cw+CN w^2+\frac{C}{N}$. Thus
$$
\begin{cases}
\frac{d \eta}{dt}(t) \geq - g(\eta(t)), \\
\eta(T)=w(T).
\end{cases}
$$
A standard argument from the basic theory of differential inequalities can now be used to prove that
$\eta(t) \leq v(t)$ for $0 \leq t \leq T$, if $v(t)$
is the solution of
$$
\begin{cases}
\frac{d v}{dt}(t) = - g(v(t)), \\
v(T)=w(T).
\end{cases}
$$
This last result can be combined with Lemma \ref{tfc}, the inequality  \eqref{lbh} which means $w(T)\leq \frac{C}N$ 
 and the inequality (\ref{blabla}), to
prove that $w(t) \leq \frac{2C}N$ for all $0 \leq t \leq T$, which ends the proof of the Proposition.
\end{proof}


\subsection{Convergence}
\label{convsubsec}

In this section we  prove Theorem \ref{teoconv}, which implies the convergence of both distribution and value function of the $N+1$-player game to the mean field game, for small times.

Let $\theta_0 = (\theta_0^{1}, \theta_0^{2} \hdots, \theta_0^{d}) \in \mathcal{S}^d$  be given.
We start by assuming that  at the initial time the $N$ players distinct from the reference player are randomly assigned  states $1,2,\hdots,d$ independently according to the initial distribution $\theta_0$ (i.e. choosing state $k$ with probability $\theta_0^{k}$).
Therefore, $\bn_0$ is a random vector of $\mathbb{Z}^d$ that follows a multinomial distribution with parameters  $N$ and $\theta_0$.

We will write $\bn_t^l$ for the l-th coordinate of $\bn_t$, which means the number of players (distinct from the reference player) that are in state $l$ at time $t$.  

The norm we use for vectors of $\re^d$, in this section, is the norm  $\|v\|=\max\{|v^1|,|v^2|,...,|v^d|\}$, where $|v^i|$ is the absolute value of the $i$-th coordinate of $v$.

The main result is the following:
\begin{theorem}\label{teoconv} Let $T^*$ be as in Proposition \ref{lipbounds}. There exists a constant $C$, independent of $N$, for which, if $T<T^*$, satisfies $\rho = TC <1$,
then
$$V_N(t)+W_N(t) \leq \frac{C}{1-\rho} \frac{1}{N}\, \,,$$ for all $ t \in [0,T],$
where
$$
V_N(t) \equiv
\mathbb{E}\left[ \left\| \frac{\bn_t}{N} - \theta(t) \right\|^2  \right]\,,
$$
and
$$
W_N(t) \equiv \mathbb{E}\left[ \left\|\, u(t) - u^N_{\bn_t}\left(t\right) \right\|^2  \right]\,,
$$
where the pair $\theta(t)$ and  $u=u(t)$ is the solution of the MFG game  (\ref{PVIT}),
 and $u^{N}=u^{N}(t)$ is the value function of the $N+1$-player game, i.e., the solution of  game \eqref{eqhj}.
\end{theorem}

Before proving Theorem \ref{teoconv} we need two Lemmas. Let
\begin{equation}\label{VNQN}
\begin{cases}
 V_N(l,t) \equiv \mathbb{E}\left[ \left( \frac{\bn^l_t}{N} - \theta^l(t) \right)^2  \right]\,, \\
   \\
 W_N(l,t) \equiv \mathbb{E}\left[ \left( u^l(t) - u^{N,l}_{\bn_t}(t) \right)^2  \right]\,. 
\end{cases}
\end{equation}
We have
\begin{equation}\label{vnt}
 V_N(t)= \max\{ V_N(1,t),\hdots, V_N(d,t)\} \hspace{.2cm} \mbox{and} \hspace{.2cm}  W_N(t)= \max\{ W_N(1,t),\hdots, W_N(d,t)\},
 \end{equation}
and
\begin{equation}\label{var}
    V_N(l,0)= \mbox{Var}\left[ \frac{\bn^l_0}{N}\right] = \frac{ \theta_0^{l}(1- \theta_0^{l})}{N}\,,
\end{equation}
 because $\bn^l_0$ is the sum of $N$ independent and identically distributed random variables, each of them having Bernoulli distribution with parameter $\theta_0^{l}$.

An important tool for proving Lemmas \ref{conv1} and \ref{conv2} is again the Dynkin Formula, now adapted to the present situation: define the infinitesimal generator of the process $(\bi,\bn)$ acting on a function   $\varphi:I_d \times \mathcal{S}^d_N \times [0,+\infty) \rar \Rr$, $C^1$ in the last variable,      by
\begin{align}
\label{generator-conv}  A^{\alpha}\varphi(i,n,s) = &\sum_j \alpha_{ij}^{N,i}[\varphi(j,n,s)-\varphi(i,n,s) ]+  
\sum_{kj}  n^k \alpha^{N,i}_{kj}  [ \varphi(i,{n+e_{jk}},s)-\varphi(i,n,s)   ], 
 \end{align}
where
\begin{equation}\label{controlN}
    \alpha^{N,i}_{kj}= \alpha^*_j\left(\Delta_k u^N_{n+e_{ik}},\frac{n+e_{ik}}N,k\right),
\end{equation}
is the transition rate from state $k$ to state $j$ in for equilibrium solutions of the $N+1$-player game, as in section  \ref{eqsols}. 


Then for any $t<T$,
\begin{equation}\label{Dynkin-conv}  \mathbb{E} \left[\varphi({\bi_T},{\bn_T},T)-\varphi({\bi_t},{\bn_t},t)\right]
= \mathbb{E} \left[ \int_t^T
\frac{d\varphi}{dt} (\bi,\bn,s) + A^{\alpha}\varphi({\bi},{\bn},s) ds  \right]\,.
 \end{equation}
Note that in the right hand side of the equation above, the processes $\bi$ and $\bn$ are evaluated at time $s$.

We will also denote by
 $$
 \alpha_{ij}= \alpha^*_j(\Delta_i u, \theta ,i)
 $$
 the transition rate from state $i$ to state $j$ in the equilibrium solutions of the mean field game as in section \ref{nash_eq}.

\begin{lemma}\label{conv1} Let $T^*$ be as in Proposition \ref{lipbounds}, and suppose $T<T^*$.
There exists $C_1>0$ such that
$$V_N(t) \leq \int_0^t C_1 (V_N(s)+ W_N(s)) ds + \frac{C_1}N.$$
\end{lemma}
\begin{proof}
Using Dynkin's Formula (\ref{Dynkin-conv}) with $\varphi_l(i,n,s)=\left(\frac{n^l}N-\theta^l(s)\right)^2$, and \eqref{var}, we have
$$ V_N(l,t) -  \frac{\theta_0^{l}(1-\theta_0^{l})}{N} =
 \mathbb{E} \int_0^t \bigg(\omega_{N,l}(s) + \varsigma_{N,l}(s)\bigg)ds\,,$$
where
 $$
 \varsigma_{N,l}(s) = \frac{d \varphi_l}{dt}(i,\bn,s)=
 - 2  \left( \frac{\bn^l}{N}-\theta^l \right)
\sum_k \alpha_{kl}\theta^k,
$$
and
$$
\omega_{N,l}(s)=\sum_{k}\sum_j \bn^k \alpha^{N,i}_{kj}\bigg[\varphi_l(\bn+e_{jk},s)-\varphi_l(\bn,s)  \bigg].
$$
Note that  $\varphi_l(i,\bn,s)$ just depend on $\bn^l$ and $s$.
Therefore $ \varphi_l(\bn+e_{jk},s)=\varphi_l(\bn,s)   $ if both $j\neq l $ and $k\neq l$.
Hence
\begin{align*}\omega_{N,l}(s)&=\sum_{k\in I_d, j=l} \bn^k \alpha^{N,i}_{kl}\bigg[\varphi_l(\bn+e_{lk},s)-\varphi_l(\bn,s)  \bigg]+\sum_{j\in I_d, k=l} \bn^l \alpha^{N,i}_{lj}\bigg[\varphi_l(\bn+e_{jl},s)-\varphi_l(\bn,s)  \bigg]\\
&=\sum_{k\neq l} \bn^k \alpha^{N,i}_{kl}\bigg[\bigg( \frac{\bn^l+1}{N}-\theta^l\bigg)^2  -\bigg( \frac{\bn^l}{N}-\theta^l  \bigg)^2\bigg]\\
&+\sum_{j\neq l} \bn^l \alpha^{N,i}_{lj}\bigg[\bigg( \frac{\bn^l-1}{N}-\theta^l\bigg)^2  -\bigg( \frac{\bn^l}{N}-\theta^l  \bigg)^2\bigg]\\
&=\bigg(2\bigg(\frac{\bn^l}{N}-{\theta^l}\bigg)+\frac 1 N   \bigg) \sum_{k\neq l} \frac{\bn^k}N \alpha^{N,i}_{kl}+\bigg(2\bigg(-\frac{\bn^l}{N}+{\theta^l}\bigg)+\frac 1 N   \bigg) \frac{\bn^l}N\sum_{j\neq l}  \alpha^{N,i}_{lj}\\
&=\bigg(2\bigg(\frac{\bn^l}{N}-{\theta^l}\bigg)+\frac 1 N   \bigg) \sum_{k\neq l} \frac{\bn^k}N \alpha^{N,i}_{kl}+\bigg(2\bigg(\frac{\bn^l}{N}-{\theta^l}\bigg)-\frac 1 N   \bigg)  \frac{\bn^l}N \alpha^{N,i}_{ll}\\
&\leq  2\bigg(\frac{\bn^l}{N}-{\theta^l}  \bigg) \sum_{k\in I_d} \frac{\bn^k}N \alpha^{N,i}_{kl}+ \frac{\tilde C}{N},
\end{align*}
where in the inequality above we used the fact that $\alpha^{N,i}_{kl}$ is bounded (with bounds that do not depend on $N$ - here we are using that $\alpha^*$ is Lipschitz and $\Delta_i u_n$ is bounded - see Propositions \ref{Lipschitz0} and \ref{lipbounds}).
Now
\begin{align*}
 \varsigma_{N,l}(s)+\omega_{N,l}(s) \leq&  2  \left( \frac{\bn^l}{N}-\theta^l \right)
\sum_k \bigg[ \frac{\bn^k}N \alpha^{N,i}_{kl}-\theta^k\alpha_{kl}\bigg]+ \frac{\tilde C}{N}\\
=&  2  \left( \frac{\bn^l}{N}-\theta^l \right)
\sum_k \bigg[ \frac{\bn^k}N \alpha^{N,i}_{kl} -\frac{\bn^k}N  \alpha_{kl} + \frac{\bn^k}N  \alpha_{kl} -\theta^k\alpha_{kl}\bigg]+ \frac{\tilde C}{N}\\
=&  2  \left( \frac{\bn^l}{N}-\theta^l \right)
\sum_k \bigg[\frac{\bn^k}N\bigg(\alpha^{N,i}_{kl} -\alpha_{kl}\bigg) + \alpha_{kl}\bigg(\frac{\bn^k}N-\theta^k \bigg) \bigg]+ \frac{\tilde C}{N}\,.
\end{align*}

Then

\begin{align*} V_N(l,t) =&
 \mathbb{E} \int_0^t \bigg(\omega_{N,l}(s) + \varsigma_{N,l}(s)\bigg)ds
 +\frac{\theta_0^{l}(1-\theta_0^{l})}{N} \\
\leq &
  \mathbb{E} \int_0^t 2  \left( \frac{\bn^l}{N}-\theta^l \right)
\sum_k \frac{\bn^k}N\bigg(\alpha^{N,i}_{kl} -\alpha_{kl}\bigg) ds\\
+&
\mathbb{E} \int_0^t 2  \left( \frac{\bn^l}{N}-\theta^l \right)    \sum_k \alpha_{kl}\bigg(\frac{\bn^k}N-\theta^k \bigg)   ds+\frac{\tilde  C T+1/4}{N}\,.
\end{align*}

Now, using again the fact that $\alpha^*$ is Lipschitz and
\begin{equation}\label{diffopeiscontract}
    \|\Delta_k w - \Delta_k z \| \leq  \|w-z\|\,,
\end{equation}
we see that 

\begin{align*}|\alpha^{N,i}_{kl} -\alpha_{kl}|&=
\left|\alpha_l^*\left(\Delta_k u^{N}_{\bn+e_{ik}},\frac{\bn+e_{ik}}N,k\right)- \alpha_l^*(\Delta_k u, \theta ,k)\right| \\
&\leq K\bigg( \left\|\theta-\frac{\bn+e_{ik}}N\right\|+\|u^N_{\bn+e_{ik}}- u\| \bigg)\\
&\leq K\bigg( \left\|\theta-\frac{\bn}N\right\|+\frac 2 N+\|u^N_{\bn+e_{ik}}- u^N_{\bn}\|+\|u^N_{\bn}- u\| \bigg)\\
&\leq
K\bigg( \left\|\theta-\frac{\bn}N\right\|+\frac {2+2C_0} N +\|u^N_{\bn}- u\| \bigg)\,,
\end{align*}
where in the last equality we used the gradient estimates of Proposition \ref{lipbounds}.

Therefore
\begin{align*} V_N(l,t) \leq &
  2K \;\mathbb{E} \int_0^t   \left| \frac{\bn^l}{N}-\theta^l \right|
\bigg( \left\|\theta-\frac{\bn}N\right\|+\frac {2+2C_0} N +\|u^N_{\bn}- u\| \bigg) ds+\\
+ & \mathbb{E} \int_0^t 2  \left( \frac{\bn^l}{N}-\theta^l \right)  \sum_k \alpha_{kl}\bigg(\frac{\bn^k}N-\theta^k \bigg)  ds + \frac {\tilde C T+1/4} N   \\
\leq &
2K \;\mathbb{E} \int_0^t   \left| \frac{\bn^l}{N}-\theta^l \right|
\bigg( \left\|\theta-\frac{\bn}N\right\| +\|u^N_{\bn}- u\| \bigg) ds+\\
+ & C_3 \mathbb{E} \int_0^t 2  \left| \frac{\bn^l}{N}-\theta^l \right|   \sum_k \bigg|\frac{\bn^k}N-\theta^k \bigg|   ds + \frac {C_4} N   \,
\end{align*}
where we used the fact that $\alpha_{kl}$ is bounded by a constant $C_3$,
and $C_4=\tilde C T +1/4 +2T+2C_0T$.

Now
\begin{align*} V_N(l,t) \leq &
2K \;\mathbb{E} \int_0^t
\bigg( \left\|\theta-\frac{\bn}N\right\|^2 + \|u^N_{\bn}- u\|  \left\|\theta-\frac{\bn}N\right\|\bigg) ds+\\
+ & 2 d C_3 \mathbb{E} \int_0^t \left\|\theta-\frac{\bn}N\right\|^2  ds + \frac {C_4} N   \,.
\end{align*}
Finally using $ab < a^2 + b^2$ and \eqref{vnt}, we have
$$
V_N(t) \leq \int_0^t C_1 (V_N(s)+ W_N(s)) ds + \frac{C_1}N\,,
$$
where $C_1= 3 \max\{2  K, 2 d C_3, C_4\}$.

\end{proof}
\vspace{.3cm}

\begin{lemma}\label{conv2} Let $T^*$ be as in Proposition \ref{lipbounds}, and suppose $T<T^*$. There exists $C_2>0$ such that
$$W_N(t) \leq \int_t^T C_2 (V_N(s)+ W_N(s)) ds + \frac{C_2}N. $$
\end{lemma}
\begin{proof}

Using Dynkin formula
(\ref{Dynkin-conv}) with $\varphi_l(i,n,s)=\Big(u^{N,l}_{n}(s)-u^l(s)\Big)^2$, and equations
(\ref{eqhj}) and (\ref{PVIT}),
 we have
\begin{align*}
&W_N(l,t)-W_N(l,T)=-\mathbb{E}\bigg[\Big(u^{N,l}_{\bn}(t)-u^l(t)\Big)^2\bigg]+\mathbb{E}\bigg[\Big(u^{N,l}_{\bn}(T)-u^l(T)\Big)^2\bigg]\\
=& \mathbb{E} \int_t^T 2 (u^{N,l}_{\bn}-u^l) \frac{d}{ds} (u^{N,l}_{\bn}-u^l) ds  +\mathbb{E} \int_t^T\sum_{jk} \bn^k \alpha^{N,i}_{kj} \bigg[\varphi_l(i,\bn+e_{jk},s)-\varphi_l(i,\bn,s)\bigg]
ds \\
=& \mathbb{E} \int_t^T 2 (u^{N,l}_{\bn}-u^l) \left( -\sum_{jk} \bn^k \alpha^{N,i}_{kj} \Big(u^{N,l}_{\bn+e_{jk}}-u^{N,l}_{\bn}\Big)-h\Big(\Delta_l u_{\bn}^N,\frac{\bn}N,l\Big) +h (\Delta_l u,\theta,l)   \right) ds\\
&\qquad + \mathbb{E} \int_t^T\sum_{jk} \bn^k \alpha^{N,i}_{kj} \bigg[\Big(u^{N,l}_{\bn+e_{jk}}-u^l\Big)^2-\Big(u^{N,l}_{\bn}-u^l\Big)^2\bigg]
ds\\
=&\mathbb{E} \int_t^T \sum_{jk} \bn^k \alpha^{N,i}_{kj} \bigg[ -2 \Big(u^{N,l}_{\bn}-u^l\Big)   \Big(u^{N,l}_{\bn+e_{jk}}-u^{N,l}_{\bn}\Big)+ \Big(u^{N,l}_{\bn+e_{jk}}-u^l\Big)^2-\Big(u^{N,l}_{\bn}-u^l\Big)^2\bigg] ds\\
&\qquad + \mathbb{E} \int_t^T \Big(2 (u^{N,l}_{\bn}-u^l) \Big) \Big(h (\Delta_l u,\theta,l)-h\Big(\Delta_l u^N_{\bn},\frac{\bn}N,l\Big) \Big)
ds\\
=&\mathbb{E} \int_t^T \sum_{jk} \bn^k \alpha^{N,i}_{kj} \Big(u^{N,l}_{\bn+e_{jk}}-u^{N,l}_{\bn}   \Big)^2   ds\\
&\qquad + \mathbb{E} \int_t^T \Big(2 (u^{N,l}_{\bn}-u^l) \Big) \Big(h (\Delta_l u,\theta,l)-h\Big(\Delta_l u^N_{\bn},\frac{\bn}N,l\Big) \Big)
ds.
\end{align*}
In the last equation we used the fact that
$$  -2 \Big(u^{N,l}_{\bn}-u^l\Big)   \Big(u^{N,l}_{\bn+e_{jk}}-u^{N,l}_{\bn}\Big)+ \Big(u^{N,l}_{\bn+e_{jk}}-u^l\Big)^2-\Big(u^{N,l}_{\bn}-u^l\Big)^2=\Big(u^{N,l}_{\bn+e_{jk}}-u^{N,l}_{\bn}   \Big)^2.  $$

Now, using the gradient estimates  from \S \ref{uest}, Proposition \ref{lipbounds}, we have that
$$\alpha^{N,i}_{kj} \Big(u^{N,l}_{\bn+e_{jk}}-u^{N,l}_{\bn}   \Big)^2  < \frac{K_2}{N^2},$$ which implies $$\sum_{jk} \bn^k \alpha^{N,i}_{kj} \Big(u^{N,l}_{\bn+e_{jk}}-u^{N,l}_{\bn}   \Big)^2   < \frac{d K_2}N\,. $$
For the same reason we have that $W_N(T)$ is bounded by $\frac{K_3}N$,
which implies
$$W_N(t) \leq \frac{K_4}N +
2 \mathbb{E}\int_t^T
\left(h (\Delta_l u,\theta,l)-h\big(\Delta_l u^N_{\bn},\frac{\bn}N,l\big)\right)\Big(u^{N,l}_{\bn}-u^l\Big)ds\,.$$
Using the fact that $h$ is Lipschitz  in both variables, with Lipschitz constant uniform (since $\Delta u$ is bounded) 
and \eqref{diffopeiscontract}
we see that 
\begin{align*}
 h(\Delta_l u, \theta ,l)-h\left(\Delta_l u^N_{\bn},\frac{\bn}N,l\right)&
< K\bigg( \left\|\theta-\frac{\bn}N\right\|+\|u_{\bn}^N- u\| \bigg)\,.
\end{align*}

Therefore, using $u^{N,l}_{\bn}(s)-u^l(s)\leq \|u^N_{\bn}- u\|$ and again $ab < a^2+b^2$, we have
$$ W_N(t) \leq \frac{K_4}N  + K_5 \int_t^T
 V_N(s) + W_N(s)  ds\,,$$
 which ends the proof.
\end{proof}

Now we can  prove our main result that establishes the convergence
of the $N+1$-player game to the mean field model as $N\to \infty$.

\bigskip

{\bf Proof of Theorem \ref{teoconv}:}

Define $C=C_1+C_2$.
Adding both inequalities given in the two last Lemmas, we have
$$W_N(t)+V_N(t) \leq C \int_0^T (V_N(s)+W_N(s))ds + \frac C N\,.$$
Now suppose $\rho = TC <1$.
Defining
\[
W_N+V_N = \max_{0 \leq t \leq T}W_N(t)+V_N(t),
\]
we have
$$W_N+V_N \leq \rho (W_N+V_N) + \frac C N, $$
which proves the Theorem \ref{teoconv}.
\hfill \cqd

\section{Potential mean field games}\label{lastsection}

An important class of examples  are potential mean field games, which have additional structures that can be used to deduct further properties. In these  mean field games  $h$ has the form


\begin{equation}\label{separatedh}
    h(z,\theta,i)=\tilde h(z,i)+f^i(\theta)
\end{equation}
where $\tilde h:\re^d\times I_d \rar \re$ and $f: \re^d\times I_d \rar \re$ is the gradient of a convex function.
More precisely, we suppose that there exists a convex function $F:\re^d \rar \re$ such that $\nabla_{\theta}F=f(\cdot,\theta)$.

\subsection{Hamiltonian and Lagrangian formulations}

Let $H:\re^{2d}\rar \re$ be given by
\begin{align}\label{hamilt}
    H(u,\theta)&=\sum_i \theta^i \tilde h(\Delta_i u,i) + F(\theta)
\\  & = \theta \cdot \tilde h(\Delta_{\cdot}u,\cdot)+F(\theta) \,.   \nonumber
\end{align}
A direct computation 
shows that \eqref{PVIT} can be written as
\begin{equation}\label{Hamiltequat}
    \left\{
      \begin{array}{l}
        \frac{\partial H}{\partial u^j} =  \dot \theta^j\,,\\
        \\
        \frac{\partial H}{\partial \theta^j} = - \dot u^j.
      \end{array}
    \right.
\end{equation}
This means the flow generated by  equation \eqref{PVIT} is Hamiltonian.
In addition to the fact that  the Hamiltonian is preserved by the flow \eqref{Hamiltequat},
the special structure of the $H$, which  depends only on $\Delta_iu$, implies that $\sum_i \theta^i$ is also a conserved quantity,
which is consistent with the interpretation of $\theta$ in terms of probability distribution of players.


Given a convex function $G(p)$ we define the Legendre transform
as
\[
G^*(q)=\sup_p -q\cdot p-G(p).
\]
If $G$ is strictly convex and the previous supremum is achieved, then $q=-\nabla G(p)$, or equivalently $p=-\nabla G^*(q)$.

If the function $F$ is strictly convex in $\theta$ then the Hamiltonian $H$ is strictly convex in $\theta$.
This allow us to consider the Legendre transform
\begin{align*}
L(u, \dot u)&=\sup_{\theta} -\dot u \cdot \theta-H(u, \theta)\\
&=\sup_{\theta} -(\dot u+\tilde h)\cdot  \theta-F(\theta)=F^*(\dot u+\tilde h(\Delta_{\cdot}u,\cdot)).
\end{align*}
From this we conclude that any solution to \eqref{PVIT} is a critical point of the functional
\begin{equation}\label{cpf}
\int_0^T F^*(\dot u+\tilde h(\Delta_\cdot u,\cdot))ds.
\end{equation}
This variational problem has to be complemented by suitable boundary conditions. The initial-terminal value problem
corresponds to
\begin{align*}
\theta_0&=-\nabla F^*(\dot u(0)+\tilde h(\Delta_\cdot u(0), \cdot)),\\
u(T)&=\psi(\cdot, -\nabla F^*(\dot u(T)+\tilde h(\Delta_\cdot u(T), \cdot))).
\end{align*}
Another important boundary condition arises in planning problems. In this case the objective is to find a
terminal cost $u(T)$  which steers a initial probability distribution $\theta_0$ into a terminal probability distribution $\theta^T$.
Hence we have the following
\begin{align*}
\theta_0&=-\nabla F^*(\dot u(0)+\tilde h(\Delta_\cdot u(0), \cdot)),\\
\theta^T&=-\nabla F^*(\dot u(T)+\tilde h(\Delta_\cdot u(T), \cdot)).
\end{align*}

The variational principle \eqref{cpf} is an analog to the results in \cite{GPSM, GMorgado}.

\subsection{Two PDE's for the value function}

We will present now a PDE for the value function.
As pointed out by Lions in his course in College de France, as well as in \cite{Gueant1,Gueant2}
the value function of the mean field game can be determined by solving a PDE. For this let
 $g:\re^d\times \mathcal{S}^d \times I_d \to \re^d$ be
$$ g(u,\theta,i)=\sum_j \theta^j \alpha^*_i(\Delta_ju,\theta,j).$$
The first equation of \eqref{PVIT} is equivalent to $\frac{d}{dt}\theta^i=g(u,\theta,i)$.

Consider the PDE
\begin{equation}\label{PDE}
     -\frac{\partial U^i}{\partial t}(\theta,t) = h(U,\theta,i)+ \sum_k g(U,\theta,k) \frac{\partial U^i}{\partial \theta^k}(\theta,t) \,,
\end{equation}
where $U:I_d\times\mathcal{S}^d\times[0,T] \to \re$,
and the terminal condition
\begin{equation}\label{PDE_termcond}
    U^i(\theta,T)=\psi^i(\theta)\,.
\end{equation}

A direct computation show us that the following Proposition holds:

\begin{proposition}
Suppose $U:I_d\times\mathcal{S}^d\times[0,T] \to \re$ is a solution of \eqref{PDE} and \eqref{PDE_termcond}.
Let $\theta:[0,T]\to \mathcal{S}^d$ and $u:[0,T]\to \re^d$ be two functions such that
\begin{enumerate}
  \item the first equation of \eqref{PVIT} is satisfied, i.e. $\frac{d}{dt}\theta^i=g(u,\theta,i)$;
  \item $\theta(0)= \theta_0$;
  \item $u^i(t)= U^i(\theta(t),t)\,.$
\end{enumerate}
Then $u$ satisfies the second equation of \eqref{PVIT}, i.e. $-\frac{d}{dt}u^i = h(\Delta_iu,\theta,i)$ as well as the terminal condition $u^i(T)=\psi^i(\theta(T))$. Therefore, $u$ is the value function associated to $\theta$, and so it determines a Nash equilibria for the MFG.
\end{proposition}

 As a consequence of the above Proposition, if $U$ is a solution of \eqref{PDE} and \eqref{PDE_termcond}, the initial value problem
\begin{eqnarray}
\begin{cases}
  \frac{d}{dt}\theta^i=g(U^i(\theta(t),t),\theta,i), \\
  \theta(0)=\theta_0 \label{ODEPVI},
\end{cases}
\end{eqnarray}
can be solved by the usual methods of the  ODE theory to find a Nash equilibrium $\theta$ and the associated value function $u^i(t)= U^i(\theta(t),t)\,,$
for any initial distribution $\theta_0$.
Also, the function $U$ allows one to calculate the optimal strategies for each player, at any time.
In fact the optimal switching of a player in state $i$, given the distribution $\theta$ of players, is $\alpha_j^*(\Delta_i U(\theta,t),\theta,i)$, for $1 \leq j \leq d$.


We should observe that \eqref{eqhj} can be regarded as a discretization of \eqref{PDE}. Indeed, set $\theta=\frac n N$
and assume $u_n^i\simeq U^i(\theta, t)$, for some smooth function $U$. Then, for large $N$, from \eqref{gammaeq}
we have
\begin{align*}
\frac{\gamma^{n,i}_{ kj}}{N}&= \frac{n^k}{N} \alpha_j^*\Bigg(\Delta_k u_{n+e_{ik}},\frac{n+e_{ik}}{N},k\Bigg)\,\\
&\simeq \theta^k \alpha_j^*\Bigg(\Delta_k U,\theta,k\Bigg).
\end{align*}
Furthermore,
\begin{align*}
\frac{u^i_{n+e_{jk}} - u^i_n}{1/N}&=\frac{u^i_{n+e_{jk}} -u^i_{n+e_j}+u^i_{n+e_j}-u^i_n}{1/N}\\
&\simeq -\frac{\partial U^i}{\partial \theta^k}+\frac{\partial U^i}{\partial \theta^j}.
\end{align*}
Therefore
\begin{align*}
\sum_{k,j} \gamma^{n,i}_{ kj} (u^i_{n+e_{jk}} - u^i_n  )&\simeq
\sum_{k,j} \theta^k \alpha_j^*\left(\Delta_k U,\theta,k\right)\left(\frac{\partial U^i}{\partial \theta^j}-\frac{\partial U^i}{\partial \theta^k}\right)\\
&=\sum_{k,j} \theta^k \alpha_j^*\left(\Delta_k U,\theta,k\right)\frac{\partial U^i}{\partial \theta^j},
\end{align*}
taking into account that $\sum_j\alpha_j^*=0$.
Observing that
\[
\sum_{k,j} \theta^k \alpha_j^*\left(\Delta_k U,\theta,k\right)\frac{\partial U^i}{\partial \theta^j}=
\sum_j g(U,\theta,j) \frac{\partial U^i}{\partial \theta^j},
\]
and
\[
h\left(\Delta_i u_n, \frac n N, i\right)=h\left(u_n, \frac n N, i\right)\simeq h\left( U, \theta, i\right),
\]
we conclude, from the above and \eqref{eqhj}, that
\[
-\frac{\partial U^i}{\partial t}(\theta,t) \simeq h(U,\theta,i)+ \sum_j g(U,\theta,j) \frac{\partial U^i}{\partial \theta^j}(\theta,t).
\]


For potential mean field games \eqref{PDE} can be further simplified if we  suppose that
 the terminal condition is given by a gradient
\begin{equation}\label{termcond-pot}
    U^i(\theta,T)= \nabla_{\theta^i} \Psi_T(i,\theta)\,.
\end{equation}
In this case let $\Psi$ be a solution of the  PDE
\begin{equation}\label{outrapde}
    \begin{cases}
    -\frac{\partial \Psi}{\partial t}= H\left(\nabla_{\theta} \Psi, \theta\right), \\
    \Psi(\theta,T) = \Psi_T(\theta)\,.
    \end{cases}
\end{equation}
Then a direct calculation can show that $U^i(\theta,t) = \nabla_{\theta^i} \Psi(\theta,t)$ is a solution of  \eqref{PDE} together with the terminal condition \eqref{termcond-pot}.
We should observe that the solutions to the Hamilton´s equations \eqref{Hamiltequat} are in fact characteristic curves for \eqref{outrapde}.
The Hamilton-Jacobi PDE \eqref{outrapde} was explored in \cite{Gueant1,Gueant2}  to the study of MFG problems on graphs.


\appendix

\section{Auxiliary Results}

{\bf Proof of Proposition \ref{Lipschitz0}: } 

To prove the first item we use the definition of $h$ and $\alpha^*$ and also that $\displaystyle v^i\sum_j \alpha^*_j(z,\theta,i)=0$ to get 
$$ h(z,\theta,i)+\sum_j \alpha^*_j(z,\theta,i)\,v^j= c(i,\theta,\alpha^*(z,\theta,i))+\sum_j \alpha^*_j(z,\theta,i)\,(z^j+v^j-z^i-v^i)\,.$$
 Hence by the definition of $h(z+v,\theta,i)$ we have
$$ h(z,\theta,i)+\sum_j \alpha^*_j(z,\theta,i)\,v^j\geq h(z+v,\theta,i).$$
From this \eqref{cvx} holds and we deduct that if $h$ is differentiable
\[
\alpha^*_j=\frac{\partial h(\Delta_iz,\theta,i) }{\partial z^j}.
\]

Note that item (c) is a direct corollary of item (b), since
\begin{equation}\label{hasafunctionofalphastar}
    h(p,\theta,i)=c(i,\theta,\alpha^*(p,\theta,i))+\alpha^*(p,\theta,i)\cdot p
\end{equation}
and the function $c$ is Lipschitz  in $\theta$ and differentiable in $\alpha$.

From this point on in this proof we will omit the index $i$ as it is not relevant and simplifies the notation.
To prove item (b) we will use the following inequalities, which are consequence of the uniform convexity of $c$: for all $\theta, \theta\,'\in \mathcal{S}^d$,   $\alpha\,', \alpha\in(\mathbb{R}_0^+)^d$, $\sum_k \alpha_k=\sum_k \alpha_k'=0$,  and $  p,p\,'\in \mathbb{R}^d$, we have
\begin{equation}\label{uniconv1}c(\theta,\alpha')+\alpha'\cdot p\,'\geq c(\theta,\alpha)+\alpha \cdot p\,'+(\nabla_\alpha c(\theta,\alpha)+p\,')\cdot(\alpha'-\alpha)+\gamma \|\alpha'-\alpha\|^2,
\end{equation} and because $\alpha^*(p,\theta)$ is a minimizer,
 \begin{equation}\label{uniconv2}(\nabla_\alpha c(\theta,\alpha^*(p,\theta))+p)\cdot(\alpha'-\alpha^*(p,\theta))\geq 0\,.
 \end{equation}

 We will first prove that $\alpha^*$ is uniformly Lipschitz in $p$ : for that, we
 suppose  that $\theta$ is fixed. By the definition of $\alpha^*$ and  equation \eqref{uniconv1} we have
$$ c(\alpha^*(p))+\alpha^*(p)\cdot p\,'\geq c(\alpha^*(p\,'))+\alpha^*(p\,') \cdot p\,' \geq$$
$$\geq c(\alpha^*(p))+\alpha^*(p)\cdot p\,'+(\nabla_\alpha c(\alpha^*(p))+p\,')\cdot(\alpha^*(p\,')-\alpha^*(p))+\gamma \|\alpha^*(p\,')-\alpha^*(p)\|^2,
$$ hence

$$ 0\geq (\nabla_\alpha c(\alpha^*(p))+p)\cdot(\alpha^*(p\,')-\alpha^*(p))+(p\,'-p)\cdot(\alpha^*(p\,')-\alpha^*(p))+\gamma \|\alpha^*(p\,')-\alpha^*(p)\|^2.
$$

Now using  \eqref{uniconv2} we obtain

$$ 0\geq (p\,'-p)\cdot(\alpha^*(p\,')-\alpha^*(p))+\gamma \|\alpha^*(p\,')-\alpha^*(p)\|^2.
$$
Therefore
$$\|p-p\,'\,\|\,\|\alpha^*(p\,')-\alpha^*(p)\|\geq\gamma \|\alpha^*(p\,')-\alpha^*(p)\|^2\,,
$$
which implies
$$\|\alpha^*(p\,')-\alpha^*(p)\|\leq \frac{1}{\gamma}\big\|p\,'-p\,\big\|.$$
This shows that $\alpha^*$ is uniformly Lipschitz in $p$.

Now we prove that $\alpha^*$ is Lipschitz in $\theta$: for that, we suppose that $p$ is fixed.  Again by the definition of $\alpha^*$ and  by equation \eqref{uniconv1} we have
$$c(\theta\,',\alpha^*(\theta))+ \alpha^*(\theta)\cdot p\geq c(\theta\,',\alpha^*(\theta\,'))+ \alpha^*(\theta\,')\cdot p
$$
$$\geq  c(\theta\,',\alpha^*(\theta))+ \alpha^*(\theta)\cdot p+(\nabla_\alpha c(\theta\,',\alpha^*(\theta)  )+p)\cdot(\alpha^*(\theta\,')-\alpha^*(\theta))+\gamma\|\alpha^*(\theta\,')-\alpha^*(\theta) \|^2,
$$
and then
$$0\geq (\nabla_\alpha c(\theta\,',\alpha^*(\theta)  )+p)\cdot(\alpha^*(\theta\,')-\alpha^*(\theta))+\gamma\|\alpha^*(\theta\,')-\alpha^*(\theta) \|^2.$$ Using equation \eqref{uniconv2} we get
$$0\geq \big[\nabla_\alpha c(\theta\,',\alpha^*(\theta)  )-\nabla_\alpha c(\theta,\alpha^*(\theta))\big]\cdot(\alpha^*(\theta\,')-\alpha^*(\theta))
+\gamma\|\alpha^*(\theta\,')-\alpha^*(\theta) \|^2. $$
As $\nabla_\alpha c(\theta,\alpha)$ is Lipschitz in the variable $\theta$ we have

$$K_c\|\theta\,'-\theta \|\,\|\alpha^*(\theta)-\alpha^*(\theta\,')\|\geq\gamma\|\alpha^*(\theta\,')-\alpha^*(\theta) \|^2 .$$
Therefore
$$\|\alpha^*(\theta)-\alpha^*(\theta\,')\|\leq \frac{K_c}{\gamma}\big\|\theta -\theta\,'\big\|\,, $$
which implies that $\alpha^*$ is Lipschitz in $\theta$.
\hfill \cqd
\bigskip

{\bf Proof of Theorem \ref{verif-Npl}: }

The main tool for proving Theorem \ref{verif-Npl} is once again the Dynkin Formula, now adapted to the present situation:
 suppose $\alpha$ and $\beta$ are two admissible controls.

We recall the infinitesimal generator of the process $(\bi,\bn)$ defined in \eqref{generator-conv}. We have
\begin{align}
\label{generator}  (A^{\beta,\alpha}\varphi)^i_n(s) = \sum_j \alpha_{ij}(n,s)[\varphi^j_n(s)-\varphi^i_n(s) ]+   \sum_{kj}\gamma_{\beta,kj}^{n,i} (s) [ \varphi^i_{n+e_{jk}}(s)-\varphi^i_n(s)   ]
 \end{align}
where $\gamma_\beta$ is defined by (\ref{gamma}).

We have that,
for any function $\varphi:I_d \times \mathcal{S}^d_N \times [0,+\infty) \rar \Rr$, $C^1$ in the last variable, and any $t<T$,
\begin{equation}\label{Dynkin}  \mathbb{E}^{\beta,\alpha}_{A_t(i, n)} \left[\varphi^{\bi_T}_{\bn_T}(T)-\varphi^i_n(t)\right]
= \mathbb{E}^{\beta,\alpha}_{A_t(i, n)} \left[ \int_t^T
\frac{d\varphi^{\bi_s}_{\bn_s}}{dt} (s) + (A^{\beta,\alpha}\varphi)^{\bi_s}_{\bn_s}(s) ds  \right]\,,
 \end{equation}
where $A_t(i, n)$  denotes the event $\bi_t = i$ and $\bn_t = n$.

Now we prove the theorem.
In the Dinkyn formula \eqref{Dynkin} let $\varphi=v$.
Using the terminal condition $v^{\bi_T}_{\bn_T}(T)=\psi^{\bi_T}\left(\frac{\bn_T}{N}\right)$ we have
that, for any admissible control $\alpha$,
\begin{equation}\label{termD}
  \mathbb{E}^{\beta,\alpha}_{A_t(i,n)} \left[\psi^{\bi_T}\Big(\frac{\bn_T}{N}\Big)-v^{i}_{n}(t)\right]
 = \mathbb{E}^{\beta,\alpha}_{A_t(i,n)} \left[ \int_t^T
\frac{d v^{\bi_s}_{\bn_s} }{dt}(s)  + (A^{\beta,\alpha}v)^{\bi_s}_{\bn_s}(s)  ds  \right]\,.
\end{equation}
In the next steps we will use the definition of $u$, given in \eqref{defu}, and then \eqref{termD}, \eqref{generator} , \eqref{LegendreTransform} to have
\begin{eqnarray}
  \nonumber
  u^i_n(t,\beta,\alpha) &=&\mathbb{E}^{\beta,\alpha}_{A_t(i,n)} \left[\psi^{\bi_T}\Big(\frac{\bn_T}{N}\Big)+\int_t^T c\Big(\bi_s,\frac{\bn_s}{N},\alpha_s\Big)ds\right]
\\ &=&
\nonumber
v^i_n(t) + \mathbb{E}^{\beta,\alpha}_{A_t(i,n)}  \int_t^T \left[
\frac{d v^{\bi_s}_{\bn_s} }{dt}(s)  + (A^{\beta,\alpha}v)^{\bi_s}_{\bn_s,}(s) + c\Big(\bi_s,\frac{\bn_s}{N},\alpha_s\Big)ds \right]\,
 \\ & \geq &
 \nonumber
    v^i_n(t) + \mathbb{E}^{\beta,\alpha}_{A_t(i,n)} \left[ \int_t^T
\frac{d v^{\bi_s}_{\bn_s} }{dt}(s) +
\min_{\mu\in (\re^+_0)^d} \sum_j \mu_j [v^{j}_{\bn_s}(s)-v^{\bi_s}_{\bn_s}(s)]+c\Big(\bi_s,\frac{\bn_s}{N},\mu\Big)\right.
\\ &+&
\nonumber
\left. \sum_{kj}\gamma_{\beta,kj }^{\bn_s,\bi_s}(s) [ v^{\bi_s}_{\bn_s+e_{jk}}(s)-v^{\bi_s}_{\bn_s}(s)   ] \right]ds \,
\\    &= &
\nonumber
v^i_n(t) + \mathbb{E}^{\beta,\alpha}_{A_t(i,n)} \left[ \int_t^T
\frac{d v ^{\bi_s}_{\bn_s}}{dt}(s) +
h\Big(\Delta_{\bi_s} v, \frac{\bn_s}{N},\bi_s\Big)
+ \sum_{kj}\gamma_{\beta,kj}^{\bn_s,\bi_s} ( v^{\bi_s}_{\bn_s+e_{jk}}-v^{\bi_s}_{\bn_s})
ds \right]
\\    &= &
\nonumber
v^i_n(t)\,,
\end{eqnarray}
where the last equation holds because $v$ is a solution to the  Hamilton-Jacobi equation \eqref{HJ_prim}.
Note that in this last calculation we are also proving that, for the specific control $\tilde\alpha$ given by \eqref{control},
we have $
  u^i_n(t,\beta,\tilde\alpha)=v^i_n(t)$
which show us that $\tilde\alpha$ is the optimal control and that the objective function  $u^i_n(t,\beta)$ is given by $v_n^i(t)$.

\hfill \cqd

{\bf Proof of Proposition \ref{maxpri}: }

Let $u$ be a solution to \eqref{HJ_prim}. Let $\tilde u=u+\rho (T-t)$. Then

$$
- \frac{d\tilde u_n^i}{dt}  =\rho+
 \sum_{k,j}\gamma_{\beta,kj}^{n,i} (\tilde u^i_{n+e_{jk}}-\tilde u^i_n)
+ h\left(  \Delta_i \tilde u_n,\frac n N,i\right)\,.
$$
Let $(i,n,t)$ be a minimum point of $\tilde u$ on $I_d\times \mathcal{S}^d_N \times [0,T]$.
We have $\tilde u^i_{n+e_{jk}}\geq\tilde u^i_n$. This implies
$\gamma_{\beta,kj}^{n,i} (\tilde u^i_{n+e_{jk}}-\tilde u^i_n)\geq 0$.
We also have $\tilde u^j_n(t)-\tilde u^i_n(t)\geq 0$ hence $\Delta_i \tilde u_n= (\tilde u^1_n(t)-\tilde u^i_n(t),...,\tilde u^d_n(t)-\tilde u^i_n(t))\geq 0$.
Hence
$$-\frac{d\tilde u_n^i}{dt}(t)  \geq h\left( \Delta_i \tilde u_n, \frac n N, i\right)+\rho\geq h\left(0,\frac n N, i\right)+\rho \,,$$
because the definition of $h(\Delta_i p, \theta, i)$, with $\Delta_i p\geq 0$. Furthermore, if we take $M<\rho< 2M$ we get
\[
-\frac{d\tilde u_n^i}{dt}(t)>0.
\]
This shows that the minimum of $\tilde u$ is achieved at $T$ hence
\[
u_n^i(t) \geq -\|u(T)\|_{\infty}-2M (T-t).
\]

\bigskip

Similarly, let
$(i,n,t)$ be a maximum point of $\tilde u$ on $I_d\times \mathcal{S}^d_N\times [0,T]$.
We have $\tilde u^i_{n+e_{jk}}\leq\tilde u^i_n$. This implies
$\gamma_{\beta,kj}^{n,i} (\tilde u^i_{n+e_{jk}}-\tilde u^i_n)\leq 0$.
We also have
$\Delta_i \tilde u_n\leq 0$.
Hence
$$-\frac{d\tilde u_n^i}{dt}(t)  \leq h\left(\Delta_i \tilde u_n, \frac n N, i\right)+\rho\leq h\left(0,\frac n N, i\right)+\rho \,.$$
Furthermore, if we take $-2 M < \rho <-M$ we get
\[
-\frac{d\tilde u_n^i}{dt}(t)<0.
\]
This shows that the maximum of $\tilde u$ is achieved at $T$ hence
\[
u_n^i(t)\leq \|u(T)\|_{\infty}+2M (T-t).
\]

\hfill \cqd


{\bf Proof Lemma \ref{estimgamma}: }

Recall that $\alpha^*(p, \theta, i)$ is Lipschitz in $(p, \theta)$. Let $K$ be the corresponding Lipschitz constant.
Since $\|p\|$ bounded, we have $|\alpha^*(p,.,.)|\leq C$. Then
\begin{align*}
\bigg|\gamma_{\beta,kj}^{n+e_{rs},i}-\gamma_{\beta,kj}^{n,i}\bigg|&=\bigg| ({n+e_{rs}})^k\; \alpha_j^*\Big(\Delta_k u_{n+e_{rs}+e_{ik}},\frac{n+e_{rs}+e_{ik}}{N},k\Big)-n^k\; \alpha_j^*\Big(\Delta_k u_{n+e_{ik}},\frac{n+e_{ik}}{N},k\Big)\bigg|\\
& =\bigg|\Big[ ({n+e_{rs}})^k-n^k\Big]\; \alpha_j^*\Big(\Delta_k u_{n+e_{rs}+e_{ik}},\frac{n+e_{rs}+e_{ik}}{N},k\Big)\\
&+n^k\;\bigg[\alpha_j^*\Big(\Delta_k u_{n+e_{rs}+e_{ik}},\frac{n+e_{rs}+e_{ik}}{N},k\Big)- \alpha_j^*\Big(\Delta_k u_{n+e_{ik}},\frac{n+e_{rs}+e_{ik}}{N},k\Big)\bigg]\\
&+ n^k\;\bigg[\alpha_j^*\Big(\Delta_k u_{n+e_{ik}},\frac{n+e_{rs}+e_{ik}}{N},k\Big)- \alpha_j^*\Big(\Delta_k u_{n+e_{ik}},\frac{n+e_{ik}}{N},k\Big)\bigg]                     \bigg|\\
&\leq \bigg|\alpha_j^*\Big(\Delta_k u_{n+e_{rs}+e_{ik}},\frac{n+e_{rs}+e_{ik}}{N},k\Big)    \bigg| \\
&+ N\bigg|\alpha_j^*\Big(\Delta_k u_{n+e_{rs}+e_{ik}},\frac{n+e_{rs}+e_{ik}}{N},k\Big)- \alpha_j^*\Big(\Delta_k u_{n+e_{ik}},\frac{n+e_{rs}+e_{ik}}{N},k\Big)  \bigg|\\
&+ N\bigg|\alpha_j^*\Big(\Delta_k u_{n+e_{ik}},\frac{n+e_{rs}+e_{ik}}{N},k\Big)- \alpha_j^*\Big(\Delta_k u_{n+e_{ik}},\frac{n+e_{ik}}{N},k\Big)    \bigg|\\
&\leq C +NK \big|\Delta_k(u_{n+e_{rs}+e_{ik}}-u_{n+e_{ik}}) \big|+NK \bigg|\frac{e_{rs}}{N} \bigg|\\
&\leq C +NK 2\big\|(u^i_{n+e_{rs}+e_{ik}}-u^i_{n+e_{ik}}) \big\|+C\\
&=C+CN\big\|z^{i}_{n+e_{ik},sr} \big\|\leq C+CN\max_{rs}\|z^{\cdot}_{\cdot,sr} \|_{\infty}.
\end{align*}

\hfill \cqd

{\bf Proof of Lemma \ref{contprop}: }

 Let $z$ be a solution of \eqref{unpert2}, and fix $\epsilon>0$. We define $\tilde z= z+\epsilon (t-T)$. Hence $\tilde z$ satisfies
$$-\dot {\tilde z}_m=-\epsilon+\sum_{m'\in I_d\times\Sdn}a_{mm'}(t) ({\tilde z}_{m'}-{\tilde z}_m).$$
Let $(m,t)$ be a maximum point of $\tilde z$ on $I_d\times \Sdn \times [0,T]$.
We have ${\tilde z}_{m}(t)\geq {\tilde z}_{m'}(t)$  and this implies
$a_{mm'}(t) ({\tilde z}_{m'}-{\tilde z}_m)\leq 0$ $\forall m'$.
Hence
$$-\frac{d\tilde z_m}{dt}(t)  \leq -\epsilon. $$
This shows that the maximum of $\tilde z$ is achieved at $T$. Therefore, for all $(m',t)$,
$$z_{m'}(t)+\epsilon(t-T)=\tilde z_{m'}(t)\leq \tilde z_m(T)=z_m(T)\,.$$
Letting $\epsilon\to 0$, we get
$$z_{m'}(t)\leq \max_{m} z_m(T),\; \forall \;(m',t).
$$

From this inequality we have the following conclusions:
\begin{enumerate}
\item
if $z(T)\leq 0$, we then have
$\;z_{m'}(t)\leq 0\;$, for all $(m',t)$, and so $z(t)\leq 0$;
\item for all $(m',t)$,
$$z_{m'}(t)\leq   \|z(T)\|_{\infty}.  $$
\end{enumerate}
Now we define $\tilde z= z+\epsilon (T-t)$. Hence $\tilde z$ satisfies
$$-\dot {\tilde z}_m=\epsilon+\sum_{m'\in I_d\times\Sdn}a_{mm'}(t) ({\tilde z}_{m'}-{\tilde z}_m)\,.$$

Let $(m,t)$ be a minimum point of $\tilde z$ on $I_d\times \Sdn \times [0,T]$.
We have $a_{mm'}(t) ({\tilde z}_{m'}-{\tilde z}_m)\geq 0$.
Therefore we have
$$-\frac{d\tilde z_m}{dt}(t)  \geq \epsilon.$$
This shows that the minimum of $\tilde z$ is also achieved at $T$, hence for all $(m',t)$ we have
$$z_{m'}(t)+\epsilon(T-t)=\tilde z_{m'}(t)\geq \tilde z_m(T)=z_m(T).$$
Letting $\epsilon\to 0$, we get
$\displaystyle z_{m'}(t)\geq \min_{m} z_m(T).    $
Hence
$$z_{m'}(t)\geq  - \|z(T)\|_{\infty},  $$
and therefore we have $\|z(t)\|_{\infty} \leq \|z(T)\|_{\infty}$.

\hfill \cqd

{\bf Proof of Lemma \ref{413}: }

We note  that if $t\leq s\leq T$ we have $K(t,s)K(s,T)=K(t,T)$, which implies
\[
\frac{d}{ds}\bigg( K(t,s)K(s, T)\bigg)=0.
\]
Hence, using equation \eqref{unpert4} we get
\[
-K(t,s)M(s)K(s,T)+\left(\frac{d}{ds} K(t,s)\right)  K(s,T)=0,
\]
and therefore, by taking $T=s$ we conclude that
\begin{equation}
\label{uidt}
 \frac{d}{ds} K(t,s)=K(t, s)M(s).
\end{equation}

Multiplying \eqref{difeneq} by  $K(t,s)$ and using Lemma \ref{ordpr}, we have
$$ -K(t,s)\dot z(s)\leq K(t,s)M(s) z(s) +K(t,s)f(z(s))\,.$$
Using the identity
\[
\frac{d}{ds} K(t,s)z(s)=K(t,s) \dot z(s)+ K(t, s) M(s) z(s),
\]
which follows from \eqref{uidt},
we get
\[
-\frac{d}{ds} \Big(K(t,s)z(s)\Big)  \leq  K(t, s)f(z(s)).
\]
%
%
%
%
Thus, integrating between $t$ and $T$, we have
\[
z(t)-K(t, T)z(T)\leq \int_t^T K(t, s)f(z(s)) ds.
\]

Note that if $z(t)=K(t,T)z(T)$ is a solution of \eqref{unpert2} with terminal data $z(T)=b$, then  Lemma \ref{contprop} implies that $\|z(t)\|_{\infty} \leq \|z(T)\|_{\infty}$, hence
 $\|K(t,T)z(T)\|_{\infty} \leq \|z(T)\|_{\infty} .$

Therefore for all $m\in I_d\times \Sss^d_N$ we have
\[
z_m(t)\leq  \|z(T)\|_{\infty}+\int_t^T \|f(z(s))\|_{\infty}ds.
\]

\hfill \cqd

{\bf Proof of Lemma \ref{tfc}: }

Note that (\ref{tfc_p}) implies that
  $v$ is a monotone decreasing function of $s$ and
 is equivalent to
$$\begin{cases}
\frac{ds}{d v}= \frac{-1}{Cv+C N v^2 + \frac{C}{N}}\,,\\
s(\frac{C}{N})\leq T.
\end{cases}
$$
This implies by direct integration that
$$ s\left(\frac{2C}{N}\right)
\leq T - \int_{\frac{C}{N}}^{\frac{2C}{N}} \frac{dv}{Cv+C N v^2 + \frac{C}{N}}\,. $$ Now
$$\int_{\frac{C}{N}}^{\frac{2C}{N}} \frac{dv}{Cv+C N v^2 + \frac{C}{N}} \geq
\int_{\frac{C}{N}}^{\frac{2C}{N}} \frac{N}{2C^2+4C^3+C}dv = \frac{1}{2C+4C^2+1}. $$
Therefore if we define $T^{\star}=\frac{1}{2C+4C^2+1} $, we  have that $s\left(\frac{2C}N\right) \leq 0$ if $T \leq T^{\star}$.
Hence this implies $v(0) \leq \frac{2C}N$, which yields the desired result when we take into account that $v$ is a decreasing function of $s$.

\hfill \cqd







\bibliographystyle{alpha}

\bibliography{mfg}

\end{document}